\documentclass[12pt,reqno]{amsart}
\usepackage{amsmath}
\usepackage{amsthm}
\usepackage{amssymb}
\usepackage{graphics}
\usepackage{latexsym}

\numberwithin{equation}{section}
\newtheorem{thm}{Theorem}[section]
\newtheorem{prop}[thm]{Proposition}
\newtheorem{lem}[thm]{Lemma}

\theoremstyle{remark}

\newcommand{\p}{\partial}

\title[W. P. for fifth order KdV equation]{
}

\author[T. K. Kato]{
}
\email[Takamori Kato]{d08003r@math.nagoya-u.ac.jp}


\begin{document}

\begin{center}
{\bf WELL-POSEDNESS FOR THE FIFTH ORDER KDV EQUATION}

\bigskip {\sc Takamori Kato}

\smallskip {\small Graduate School of Mathematics, Nagoya University

Chikusa-ku, Nagoya, 464-8602, Japan}

\end{center}
\begin{abstract}
We consider the Cauchy problem of the fifth order KdV equation with low regularity data. We cannot apply the iteration argument
 to this problem when initial data is given in the Sobolev space $H^s$ for any $s \in \mathbb{R}$. So we give initial data in $H^{s,a}$ equipped with the norm 
\begin{align*}
\| \varphi \|_{H^{s,a}}=\|  \langle \xi \rangle^{s-a} |\xi|^{a}  \widehat{\varphi} \|_{L_{\xi}^2}. 
\end{align*} 
Then we recover derivatives of the nonlinear term to be able to use the iteration method. 
Therefore we obtain the local well-posedness in $H^{s,a}$ in the case of 
$s \geq \max\{-1/4, -2a-2 \}$, $-3/2<a \leq -1/4$ and $(s,a) \neq (-1/4,-7/8)$. Moreover, we obtain the ill-posedness in some sense when 
$s< \max \{-1/4, -2a-2 \}$, $a \leq -3/2$ or $a >-1/4$. 
The main tool is a variant of the Fourier restriction norm method, which is based on Kishimoto's work (2009). 
\end{abstract}
\maketitle

\section{Introduction}

We consider the Cauchy problem of the following fifth order KdV equation: 
\begin{align} \label{5KdV}
\begin{cases}
& \p_t u- \p_{x}^5 u+c_1 \p_x (u^3) + c_2 \p_x (\p_x u )^2 
+c_3 \p_x (u \p_x^2 u)=0, ~~\text{in}~~\mathbb{R} \times \mathbb{R}, \\
& u(0, x)=u_0(x), \hspace{1cm} x \in \mathbb{R},
\end{cases} 
\end{align}
where $c_1, c_2, c_3 \in \mathbb{R}$ with $c_3 \neq 0$. 
Here the unknown function $u$ is assumed to be real or complex valued when we consider the local well-posedness (LWP for short)
 and to be real valued when we deal the global well-posedness. 
\begin{align} \label{Lax}
\p_t u-\p_x^5 u -10 \p_x( u^3 ) + 5 \p_{x} (\p_x u)^2  +10 \p_x (u \p_x^2 u)
=0, 
\end{align}
is completely integrable in Lax sense and has an infinite number of conservation laws. 
The fifth order KdV equation models water waves (see, for instance, \cite{Be84}, \cite{Be77}, \cite{Ol}). 
Our main aim is to prove LWP for (\ref{5KdV}) with low regularity data. The main tool is the Fourier restriction norm method introduced by Bourgain \cite{Bo}. By using the theory of complete integrability, we obtain global solutions of (\ref{Lax}) with Schwartz data and solitary waves. But this method will not work 
for the well-posedness problem of (\ref{5KdV}) generalized to the non-integrable case. So the theory of dispersive PDEs is required, such as the Fourier restriction norm method. 

We review some known results related to this problem. 
Ponce \cite{Po} proved LWP in $H^s$ for $s \geq 4$ by the compactness argument, which was improved to $s > 5/2$ by Kwon \cite{Kw08}. 
Here the Sobolev space $H^s$ is defined by the norm 
\begin{align*}
\| \varphi \|_{H^s}:= \| \langle \xi \rangle^s  \widehat{ \varphi } \|_{L_{\xi}^2},
\end{align*}
where $\langle \xi \rangle:=(1+|\xi|^2)^{1/2}$ and $ \widehat{\varphi} $ is the Fourier transform of $\varphi$. 
Kenig, Ponce and Vega \cite{KPV94} studied the Cauchy problem for the higher order dispersive equation:
 \begin{align*} 
\p_t u + \p_x^{2j+1} u +P(u,\p_x u, \cdots, \p_x^{2j} u)=0, 
\end{align*}
where $P$ is a polynomial having no constant and linear term. Using the local 
smoothing estimates established in \cite{KPV93}, they showed LWP in the weighted Sobolev space $L^2( |x|^m dx) \cap H^s$ where $s>0$ and $m \in \mathbb{N} \cup \{0\}$ are some large numbers (see also \cite{Pi}). 
When $s>j -\frac{3}{2}-\frac{1}{2j}
+ \frac{2j-1}{2r'}$ and $1< r' \leq \frac{2j}{2j-1}$ with $j \geq 2$, 
Gr\"{o}nrock \cite{Gr} proved LWP for the Cauchy problem of the $2j+1$th order KdV equation in $\widehat{H}_{s}^{r} $, which is equipped with the norm 
\begin{align*}
\| \varphi \|_{\widehat{H}_{s}^{r} }:= \| \langle \xi \rangle^s 
\widehat{\varphi} \|_{L_{\xi}^{r'}}, 
\text{ where } \frac{1}{r}+\frac{1}{r'}=1. 
\end{align*}
Namely, he obtained LWP for (\ref{5KdV}) in $\widehat{H}_{s}^r$ when $s>\frac{1}{4}+\frac{3}{2r'}$ and $1< r \leq \frac{4}{3}$. 
Moreover, Kwon \cite{Kw08} proved LWP for the Cauchy problem of the modified fifth order KdV equation,
\begin{align} \label{m5KdV}
\p_t u -\p_x^5 u-6\p_x(u^5)+10 \p_x (u (\p_x u)^2)+ 10 \p_x(u^2 \p_x^2 u) =0,
\end{align}
in critical case $H^{3/4}$ by using the $[k, \mathbb{Z}]$-multiplier norm method and the block estimates established by Tao \cite{Ta}. 

We review difficulties in this problem. We only recover two derivative losses by the smoothing effects Lemma~\ref{lem_dy_1}--\ref{lem_dy_3} below. So the nonlinear term $\p_x(u \p_x^2 u)$ has more derivatives than can be recovered by the smoothing effects. The fact implies that Picard's interaction method is not available when 
initial data is given in $H^s$ for any $s \in \mathbb{R}$, which causes the strong interaction between high and low frequencies data. This type of phenomenon is observed in the Benjamin-Ono equation and the Kadomtsev-Petviashvili-I (KP-I) equation. 
In \cite{MST01} and \cite{MST02}, Molinet, Saut and Tzvetkov showed the data-to-solution maps of these equations fail to be $C^2$. 
Furthermore, in \cite{KT05} and \cite{KT08}, Koch and Tzvetkov proved these maps cannot be uniformly continuous. 
Using the similar argument to \cite{MST01} or \cite{MST02}, we prove that (\ref{5KdV}) the flow map fails to be $C^2$. We first define the quadratic term of the Taylor expansion of the data-to-solution map as 
\begin{align} \label{qu_def}
A_2(u_0)(t) =-c_2 \int_{0}^{t} U(t-s) \p_x( \p_x u_1(s) )^2 ds 
- c_3 \int_0^t U(t-s) \p_x (u_1(s) \p_x^2 u_1(s) ) ds.
\end{align}
where $U(t):=e^{t \p_x^5}$ and $u_1(t):=U(t) u_0$. Next, we put the sequence of initial data $\{ \phi_N \}_{N=1}^{\infty } \in H^{\infty}$ as follows:
\begin{align} \label{cou_ini}
\widehat{\phi_N} (\xi) =N^{-s+2}~\chi_{[N-N^{-4},N+N^{-4}]}(\xi)
+N^{2} \chi_{[N^{-4}/2,N^{-4}]} (\xi),
\end{align} 
for $N \gg 1$. Clearly, $\| \phi_N \|_{H^s} \sim 1 $. Substituting (\ref{cou_ini}) into (\ref{qu_def}), 
\begin{align*}
 \| A_2 (\phi_N) (t) \|_{H^{s}} \geq C N,
\end{align*}
for $|t|$ bounded, which implies the flow map, $H^s \ni u_0 \mapsto u(t) \in H^s$, cannot be $C^2$ for any $s \in \mathbb{R}$ by the general argument in \cite{Ho}. Therefore the iteration method is not available. Moreover, we remark that the modified fifth order KdV equation (\ref{m5KdV}) is linked with the fifth order KdV equation (\ref{Lax}) through the Miura transform $v \mapsto u= \alpha \p_x v + \beta v^2$ for some constants $\alpha, \beta$. If $v$ is a smooth solution of (\ref{m5KdV}), then $u$ solves (\ref{Lax}). But (\ref{5KdV}) is a non-integrable equation so that it seems unable to apply the Miura transform. 

To avoid these difficulties, we change the space in which initial data is given as follows: 
\begin{align*}
H^{s,a}(\mathbb{R}):=\bigl\{ u \in \mathcal{Z}'(\mathbb{R})~;~ 
\| u \|_{H^{s,a}}:=\| \langle \xi \rangle^{s-a} |\xi|^a \widehat{u}  \|_{L_{\xi}^2} < \infty \bigr\},
\end{align*}
where 
$\mathcal{Z}'(\mathbb{R}^n)$ denotes the dual space of
\begin{align*}
\mathcal{Z}(\mathbb{R}^n):=\bigl\{ u \in \mathcal{S}(\mathbb{R}^n)~;~ 
D^{\alpha} \mathcal{F} u(0)=0 \text{ for every multi-index }\alpha  \bigr\}.
\end{align*}
For the details of $\mathcal{Z}(\mathbb{R})$, see e.g. pp. 237 in \cite{Tr}.  Note that we can recover more derivatives of the nonlinear term $\p_x (u \p_x^2 u)$ in the interaction between high and low frequencies data when $a<0$. Therefore the iteration method works in the case of 
\begin{align} \label{co_op}
s \geq \max \bigl\{-\frac{1}{4}, -2a-2 \bigr\}, \hspace{0.3cm}
-\frac{3}{2}< a \leq -\frac{1}{4} \text{ and } 
(s,a) \neq (-\frac{1}{4}, -\frac{7}{8}),
\end{align}
and we obtain the well-posedness result in $H^{s,a}$ as follows.
\begin{thm} \label{thm_well}
 Let $s,a$ satisfy (\ref{co_op}). Then 
(\ref{5KdV}) is locally well-posed in $H^{s,a}$. 
\end{thm}

If we assume that $u$ is real valued and 
\begin{align} \label{co_gl}
c_1=-\frac{2}{5} \alpha,~~~ c_2= \alpha,~~~c_3=2 \alpha \hspace{0.3cm} \text{for} \hspace{0.3cm}
\alpha \in \mathbb{R} \setminus \{ 0 \},
\end{align}
then two conserved quantities, 
\begin{align*}
\int u^2 dx, \hspace{0.5cm} \int (\p_x u)^2 + \frac{2}{5} \alpha u^3 dx,
\end{align*}
holds. By using these, we obtain a priori estimate as follows.  
\begin{prop} \label{prop_apr}
Let $u$ be a real valued solution to (\ref{5KdV}) with (\ref{co_gl}).
Then, for $-1 \leq a  \leq -1/4$, we obtain 
\begin{align} \label{apr-es}
\sup_{0 \leq t \leq T} \| u (t,\cdot)  \|_{H^{1,a}}^2 & \leq C 
\bigl\{ \| u_0 \|_{H^{1,a}}^2 + \| u_0 \|_{L^2}^{10/3} +
T^{4/3} \bigl( \|u_0 \|_{H^1}^{10/3} +\| u_0 \|_{L^2}^{5} \bigr) \bigr\}.
\end{align}
\end{prop}

By this proposition, we extend the local-in-time solutions obtained by Theorem~\ref{thm_well} to time global ones. 
\begin{thm} \label{thm_well-2}
Let $s \geq 1$ and $-1 \leq a \leq -1/4$. Then (\ref{5KdV}) with (\ref{co_gl}) is globally well-posed in $H^{s,a} $.
\end{thm}

We put $s_a=-2a-2$ and $B_r(\mathcal{X}):= \{ u \in \mathcal{X}~;~ \| u \|_{\mathcal{X}} \leq r  \}$ for a Banach space $\mathcal{X}$. We prove the ill-posedness in the following sense when $s< \max\{-1/4,-2a-2 \}$, $a \leq -3/2$ or $a>-1/4$.

\begin{thm} \label{thm_ill}

\vspace{0.5em}

\noindent
(i) Let $r>1$, $-3/2< a < -7/8$ and $c_2 \neq c_3$. Then, from Proposition~\ref{prop_well} below, there exist $T>0$ and the flow map 
for (\ref{5KdV}) $B_r(H^{s_a,a}) \ni u_0 \mapsto u(t) \in H^{s_a,a} $ for any $t \in (0,T]$. 
Then the flow map is discontinuous on $B_r(H^{s_a,a})$ (with $H^{s,a}$ topology) to $H^{s_a,a}$ (with $H^{s,a}$ topology) for any $s<s_a$. 

\vspace{0.5em}

\noindent
(ii) Let $s < -2a-2$, $a \leq -3/2$ or $a >-1/4$. Then there is no $T>0$ 
such that for (\ref{5KdV}) with $c_2 \neq c_3$ , $u_0 \mapsto u(t)$, is 
$C^2$ as a map from $B_r(H^{s,a})$ to $H^{s,a}$ for any $t \in (0,T]$ .

\vspace{0.5em}

\noindent
(iii) Let $s <-1/4$, $a \in \mathbb{R}$ and $c_1 \neq \frac{1}{5} c_3 (c_3-c_2)$. Then there is no $T>0$ such that 
the flow map for (\ref{5KdV}), $u_0 \mapsto u(t)$, is $C^3$ 
as a map from $B_r(H^{s,a})$ to $H^{s,a}$ for any $t \in (0,T]$.
\end{thm}

\noindent
\textbf{Remark.} (i) We do not know weather LWP for (\ref{5KdV}) holds or not 
in $H^{-1/4,-7/8}$. \\
(ii) 
From Theorems~\ref{thm_well} and \ref{thm_ill}, (\ref{5KdV}) is locally well-posed in $\dot{H}^{-1/4}$ and ill-posed in some sense 
for $s \neq -1/4$. 
\vspace{0.5em}

The main idea is how to define the function space to construct the solution of (\ref{5KdV}). 
The bilinear estimates of the nonlinear term $\p_x(u \p_x^2 u)$ plays an important role to prove Theorem~\ref{thm_well}. We introduce the Bourgain space 
$\hat{X}^{s,a,b}$ in the case of (\ref{5KdV}) as follows:  
\begin{align*}
\hat{X}^{s,a,b}:=\{ f \in \mathcal{Z}'(\mathbb{R}^2)~;~\| f \|_{\hat{X}^{s,a,b}}:=\| \langle \xi \rangle^{s-a} |\xi|^a \langle \tau-\xi^5 \rangle^b f 
\|_{L_{\tau,\xi}^2} < \infty \}.
\end{align*}
We consider the bilinear estimate of the nonlinear term $\p_x(u \p_x^2 u)$ in the Bourgain space $\hat{X}^{s,a,b}$ as follows:
\begin{align} \label{BE-3}
 \| \xi (\xi^2 f)* g \|_{\hat{X}^{s,a,b-1}}
\leq C \| f \|_{\hat{X}^{s,a,b}} \| g \|_{\hat{X}^{s,a,b}}.
\end{align} 
Here we remark, from Examples 1--3 in Appendix, 
(\ref{BE-3}) fails for any $b \in \mathbb{R}$ when
\begin{align} \label{co_cr_1}
& s=-\frac{1}{4}, ~~ -\frac{7}{8} < a \leq -\frac{1}{4}, \\
\label{co_cr_2}
& s=-\frac{1}{4}+\varepsilon_1,~~a=-\frac{7}{8} \text{ and }
 s=-2a-2,~~ -\frac{27}{28} <a<-\frac{7}{8}.
\end{align}
where $\varepsilon_1$ is a sufficiently small number such that $0 < \varepsilon_1 \leq s+1/4$. 
Therefore the standard argument of the Fourier restriction norm method does not work for (\ref{co_cr_1})--(\ref{co_cr_2}). To overcome this difficulty, we make a modification on the Bourgain space to establish the bilinear estimates when (\ref{co_cr_1})--(\ref{co_cr_2}). 
An idea of a modification of the Bourgain space is introduced by Bejenaru-Tao \cite{BT}. They considered the Cauchy problem of the Schr\"{o}dinger equation with the nonlinearity $u^2$ and 
proved LWP in critical case $H^{-1}(\mathbb{R})$. We mention how to modify the Bourgian space $\hat{X}^{s,a,b}$. Here we consider the typical counterexamples of the bilinear estimate to find a suitable function space. From Example 3 in Appendix, we have to take $b=1/2$ near the curve $\displaystyle \bigl\{\tau= \frac{\xi^5}{16}\text{ and } |\xi| \geq 1  \bigr\}$ to obtain (\ref{BE-3}) for (\ref{co_cr_1}). Thus 
we modify the Bourgian norm in the high frequency part $\{ |\xi| \geq 1 \}$ as follow:
\begin{align*}
\| f \|_{\hat{X}_{(2,1)}^{s,1/2}}:=
\bigl\| \bigl\{ \| \langle \xi \rangle^s \langle \tau-\xi^5 \rangle^{1/2} f 
\|_{L_{\tau,\xi}^2(A_j \cap B_k)} \bigr\}_{j,k \geq 0} \bigr\|_{l_j^{2} l_k^{1}}.
\end{align*}
where $A_j$, $B_k$ are two dyadic decompositions of $\mathbb{R}^2$ as follows:
\begin{align*}
A_j:= & \bigl\{ (\tau,\xi) \in \mathbb{R}^2~;~ 
2^{j} \leq \langle \xi \rangle < 2^{j+1} \bigr\}, \\
B_k:= & \bigl\{ (\tau,\xi) \in \mathbb{R}^2~;~ 2^{k} \leq \langle 
\tau-\xi^5 \rangle <2^{k+1}  \bigr\},
\end{align*}
for $j,k \in \mathbb{N} \cup \{ 0 \}$. For a Banach space $\mathcal{X}$ and 
a set $\Omega \subset \mathbb{R}^{n}$, $\| \cdot  \|_{\mathcal{X}(\Omega)}$ 
denotes $\| f \|_{\mathcal{X}(\Omega)}=\| \chi_{\Omega} f \|_{\mathcal{X}}$ 
where $\chi_{\Omega}$ is the characteristic function of $\Omega$. On the other hand, 
from Examples 1 and 2 in Appendix, we need to take $b=3a/5+9/10$ on the domain 
\begin{align*}
D_0 :=\bigl\{ (\tau,\xi) \in \mathbb{R}^2~;~ |\xi| \leq 1 \text{ and }
|\tau| \sim |\xi|^{-5/3} \bigr\},
\end{align*}
so that (\ref{BE-3}) holds for (\ref{co_cr_1}). Thus we modify the Bourgain norm in the low frequency part $\{|\xi| \leq 1 \}$ as follows:
\begin{align*}
 \| f \|_{\hat{X}_L^{a}}:= 
\begin{cases}
\| f \|_{\hat{X}_L^{-1/4, 3/4}(D_1)}+
\| f \|_{\hat{X}_L^{-1/4,3/4,1}(D_2)} & \text{ for } a=-1/4, \\
\| f \|_{\hat{X}_L^{a,5a/3+9/10,1}(A_0)} & \text{ for }
-7/8 <a< -1/4 , \\
\| f \|_{\hat{X}_L^{-7/8,3/8+\varepsilon_1/2} (A_0)} & \text{ for } a=-7/8, \\
 \| f \|_{\hat{X}_L^{a,3/8+\varepsilon_2/2}(A_0)} &  \text{ for }
-3/2< a< -7/8.
\end{cases}
\end{align*}
where 
\begin{align*} 
D_1:= & \bigl\{ (\tau,\xi) \in \mathbb{R}^2~;~|\xi| \leq 1 \text{ and } 
|\tau| \geq |\xi|^{-5/3} \bigr\},\\
D_2:= & \bigl\{ (\tau,\xi) \in \mathbb{R}^2~;~|\xi| \leq 1 \text{ and } 
|\tau| \leq |\xi|^{-5/3} \bigr\}, 
\end{align*}
and $\varepsilon_2$ is a sufficiently small number such that 
$0 < \varepsilon_2 \leq -(a+7/8)$. 
Here $\hat{X}_{L}^{a,b},~ \hat{X}_L^{a,b,1}$ are defined by the norm
\begin{align*}
& \| f \|_{\hat{X}_L^{a,b}}:=
\| |\xi|^a \langle \tau-\xi^5 \rangle^b f \|_{L_{\tau,\xi}^2(A_0)}, \\
& \| f \|_{\hat{X}_L^{a,b,1}}:=
\sum_{k \geq 0} 2^{bk} \| |\xi|^a f \|_{L_{\tau,\xi}^2(A_0 \cap B_k)}.
\end{align*}
This idea of a modification of the Bourgain norm in the low frequency part is based on Kishimoto's work \cite{Ki09} which proved the well-posedness for the Cauchy problem of the KdV equation in the critical case $H^{-3/4}$ (see also \cite{TK}). From the above argument, we define the function space $\hat{Z}^{s,a}$ as follows:
\begin{align*}
\hat{Z}^{s,a}:=\bigl\{f \in \mathcal{Z}'(\mathbb{R}^2)~ ;~\|f\|_{\hat{Z}^{s,a}} :=\| p_h f \|_{\hat{X}_{(2,1)}^{s,1/2}}+ 
\| p_l f \|_{\hat{X}_L^{a}} < \infty \bigr\}.
\end{align*}
where $p_h$, $p_l$ are projection operators such that 
$(p_h f) (\xi):=f(\xi)|_{|\xi| \geq 1},~(p_l f) (\xi):=f(\xi)|_{|\xi| \leq 1}$.  Using the function space above, we obtain the following nonlinear estimates which are the main ones in this paper. 

\begin{prop} \label{prop_mult-ES}
Let $s,a$ satisfy (\ref{co_op}). Then the following estimates hold.
\begin{align}
\label{BE-1}
& \| \langle \tau-\xi^5 \rangle^{-1}  \xi (\xi f)*(\xi g) 
\|_{\hat{Z}^{s,a}} \nonumber \\
& \hspace{0.3cm} + \| \langle \xi \rangle^{s-a} |\xi|^a \langle \tau-\xi^5 \rangle^{-1} \xi (\xi f)*(\xi g)    \|_{L_{\xi}^2 L_{\tau}^1}
\leq C \| f \|_{\hat{Z}^{s,a}} \| g \|_{\hat{Z}^{s,a}} \\
\label{BE-2}
& \| \langle \tau-\xi^5 \rangle^{-1}  \xi (\xi^2 f)* g 
\|_{\hat{Z}^{s,a}} \nonumber \\
& \hspace{0.3cm} + \| \langle \xi \rangle^{s-a} |\xi|^a \langle \tau-\xi^5 \rangle^{-1} \xi (\xi^2 f)*g    \|_{L_{\xi}^2 L_{\tau}^1}
\leq C \| f \|_{\hat{Z}^{s,a}} \| g \|_{\hat{Z}^{s,a}}, \\
\label{TE-1}
& \| \langle \tau-\xi^5 \rangle^{-1}  \xi~ f*g*h 
\|_{\hat{Z}^{s,a}} \nonumber \\
& \hspace{0.3cm} + \| \langle \xi \rangle^{s-a} |\xi|^a \langle \tau-\xi^5 \rangle^{-1} \xi~ f*g*h    \|_{L_{\xi}^2 L_{\tau}^1}
\leq C \| f \|_{\hat{Z}^{s,a}} \| g \|_{\hat{Z}^{s,a}} 
\| h \|_{\hat{Z}^{s,a}}. 
\end{align}
\end{prop}
We omit the proof of (\ref{BE-1}) because we immediately obtain (\ref{BE-1}) from (\ref{BE-2}). Therefore we only prove (\ref{BE-2}) and (\ref{TE-1}) in this paper. 

We use $A \lesssim B$ to denote $A \leq C B$ for some positive constant $C$ and write $A \sim B$ to mean $A \lesssim B$ and $B \lesssim A$. 
The rest of this paper is planned as follows. In Section 2, we give some preliminary lemmas. By using these lemmas, we prove the bilinear estimate (\ref{BE-2}) in Section 3 and the trilinear estimate (\ref{TE-1}) in Section 4. In Section 5, we give the proofs of Theorem~\ref{thm_well}, Proposition~\ref{prop_apr} and Theorem~\ref{thm_ill}. 

\vspace{1em}
 
\noindent
\textbf{Acknowledgement.} The author would like to appreciate his adviser Professor Kotaro Tsugawa for many helpful conversation and encouragement and thank Dr. Kishimoto for helpful comments.

\section{Preliminaries}
In this section, we prepare the smoothing effects and linear estimates to show the main theorems and the nonlinear estimates. When we use the variables $(\tau,\xi)$, $(\tau_1,\xi_1)$ and $(\tau_2,\xi_2)$, we always assume the relation
\begin{align*}
(\tau,\xi)=(\tau_1,\xi_1) + (\tau_2,\xi_2).
\end{align*}
We mention the smoothing effects for the operator $e^{t \p_x^5 }$. 

\begin{lem} \label{lem_dy_1} 
Suppose that $f,g$ is supported on a single $A_{j}$ for $j \geq 0$. Then 
\begin{align} \label{es_dy_1-1}
\| |\xi|^{3/4} f*g \|_{L_{\tau,\xi}^2} \lesssim  
\| f \|_{ \hat{X}_{(2,1)}^{0,1/2}} \| g \|_{\hat{X}_{(2,1)}^{0,1/2}}.
\end{align}
Moreover if 
\begin{align*}
K:=\inf \{ |\xi_1-\xi_2|~;~\exists \tau_1, \tau_2 \text{ s.t. } (\tau_1,\xi_1) \in \text{\upshape supp}~f,~
(\tau_2,\xi_2) \in \text{\upshape supp}~g  \}>0, 
\end{align*}
then we have
\begin{align} \label{es_dy_1-2}
\| |\xi|^{1/2}~f*g \|_{L_{\tau,\xi}^{2}} \lesssim K^{-3/2} 
\| f \|_{\hat{X}_{(2,1)}^{0,1/2}} \| g \|_{\hat{X}_{(2,1)}^{0,1/2}}.
\end{align}
\end{lem}
\begin{proof}
It suffices to show that
\begin{align} \label{sm_1-1}
& \Bigl| \int_{\mathbb{R}^2} \int_{\mathbb{R}^2}  
f(\tau_1,\xi_1) g(\tau-\tau_1,\xi-\xi_1) 
 h(\tau,\xi) d\tau_1 d\xi_1 d\tau d\xi \Bigr| \nonumber \\
& \hspace{1.8cm} \lesssim 
 2^{k_1/2} 2^{k_2/2} \| f \|_{L_{\tau,\xi}^2}
\| g \|_{L_{\tau,\xi}^2} \| |\xi|^{-3/4} h \|_{L_{\tau,\xi}^2}
\end{align}
and 
\begin{align} \label{sm_1-2}
& \Bigl| \int_{\mathbb{R}^2} \int_{\mathbb{R}^2}  
f(\tau_1,\xi_1) g(\tau-\tau_1,\xi-\xi_1) h(\tau,\xi) 
d\tau_1 d\xi_1 d\tau d\xi \Bigr| \nonumber \\
& \hspace{1.8cm}  \lesssim 
K^{-3/2}~2^{k_1/2} 2^{k_2/2} \| f \|_{L_{\tau,\xi}^2}
\| g \|_{L_{\tau,\xi}^2} \| |\xi|^{-1/2} h \|_{L_{\tau,\xi}^2},
\end{align}
when $f$, $g$ are restricted to $B_{k_1}$, $B_{k_2}$ for $k_1, k_2 \geq 0$. That is the reason why we use (\ref{sm_1-1}) and the triangle inequality to have
\begin{align*}
& \Bigl| \int_{\mathbb{R}^2} \int_{\mathbb{R}^2} f(\tau_1,\xi_1) g(\tau-\tau_1, \xi-\xi_1) h(\tau,\xi) d\tau_1 d\xi_1 d\tau d\xi \Bigr| \\
\lesssim & 
\sum_{k_1} \sum_{k_2} \Bigl| \int_{\mathbb{R}^2} \int_{\mathbb{R}^2}  
(\chi_{B_{k_1}} f) (\tau_1,\xi_1) ( \chi_{B_{k_2}}g) (\tau-\tau_1,\xi-\xi_1) d\tau_1 d\xi_1 
 h(\tau,\xi) d\tau d\xi \Bigr| \\
 \lesssim &
\sum_{k_1} 2^{k_1/2} \| f \|_{L_{\tau,\xi}^2 (B_{k_1})} 
\sum_{k_2} 2^{k_2/2} \| g \|_{L_{\tau,\xi}^2 (B_{k_2})} 
\| |\xi|^{-3/4} h \|_{L_{\tau,\xi}^2},
\end{align*}
which implies (\ref{es_dy_1-1}). Moreover, if we assume (\ref{sm_1-2}), we obtain (\ref{es_dy_1-2}) in the same manner as above. 

We prove (\ref{sm_1-1}) and (\ref{sm_1-2}). We use Schwarz's inequality twice and Fubini's theorem to have 
\begin{align*} 
& \Bigl|\int_{\mathbb{R}^2} \int_{\mathbb{R}^2} f(\tau_1, \xi_1) g(\tau-\tau_1,\xi-\xi_1) h(\tau,\xi) d\tau_1 d\xi_1 d\tau d\xi \Bigr| \\
& \hspace{1.2cm} \lesssim \sup_{(\tau,\xi) \in \mathbb{R}^2} m(\tau,\xi)^{1/2}
\| f \|_{L_{\tau,\xi}^2} \| g \|_{L_{\tau,\xi}^2} \| h \|_{L_{\tau,\xi}^2}, 
\end{align*}
where 
\begin{align*}
m(\tau,\xi) = \int \chi_{\Lambda_1} (\tau,\xi,\tau_1, \xi_1) d\tau_1 d\xi_1, 
\end{align*}
and 
\begin{align*}
\Lambda_1:= \bigl \{ (\tau,\xi ,\tau_1,\xi_1) \in \mathbb{R}^4~;~ 
(\tau_1, \xi_1) \in \text{\upshape supp } f,~ (\tau-\tau_1,\xi-\xi_1) \in 
\text{\upshape supp } g \bigr \}. 
\end{align*}
Therefore (\ref{sm_1-1}) and (\ref{sm_1-2}) are reduced to the estimate 
\begin{align} \label{m_es}
m(\tau,\xi) \lesssim \min \bigl\{ K^{-3}~2^{k_1+k_2} |\xi|^{-1},  2^{k_1+k_2}
|\xi|^{-3/2} \bigr\},
\end{align}
and we estimate $m$. Here we fix $\tau$, $\xi \neq 0$ and consider the variation of $\xi_1$. The identity 
\begin{align*}
(\tau-\frac{\xi^5}{16})-(\tau_1-\xi_1^5)-\bigl\{ (\tau-\tau_1)-(\xi-\xi_1)^5 
\bigr\}
=\frac{5}{16} \xi (2\xi_1 -\xi)^2 \bigl\{ (2\xi_1-\xi)^2 +2 \xi^2  \bigr\}  
\end{align*}
implies
\begin{align}
& \max \biggl\{ \Bigl\{ \frac{16}{5} \frac{ |M- C(2^{k_1}+ 2^{k_2})|   }{|\xi|} + \xi^4 \Bigr\}^{1/2}-\xi^2,~ K^2 \biggr\} \nonumber \\
& \hspace{1.2cm} \leq |2\xi_1-\xi|^2 
 \leq  \Bigl\{ \frac{16}{5} \frac{ M+ C(2^{k_1}+ 2^{k_2})   }{|\xi|} + \xi^4 \Bigr\}^{1/2}-\xi^2.
\end{align}
where $M=|\tau-\xi^5/16|$ and $C$ is some positive constant. If 
\begin{align*}
K \geq \biggl\{ \Bigl\{ \frac{16}{5} \frac{|M-C (2^{K_1}+2^{K_2})|}{|\xi|}+\xi^4 \Bigr\}^{1/2}-\xi^2 \biggr
\}^{1/2}, 
\end{align*}
then the variation of $|2\xi_1- \xi|$ is bounded by 
\begin{align} \label{es_M_1}
& \biggl[ \Bigl\{ \frac{16}{5} \frac{M+C(2^{k_1} +2^{k_2})}{ |\xi| } + \xi^4 \Bigr\}^{1/2}-\xi^2 \biggr]^{1/2} -K
=  \frac{  M_{\tau,\xi}^{1/2} -(K^2 +\xi^2)}
{ (M_{\tau,\xi}^{1/2}-\xi^2)^{1/2} + K} \nonumber \\
& \hspace{1.2cm} \leq  \frac{ \frac{32C}{5 |\xi|} (2^{k_1} +2^{k_2})}
{\bigl\{ (M_{\tau,\xi}^{1/2}-\xi^2 )^{1/2}+K \bigr\} \bigl\{ M_{\tau,\xi}^{1/2}+ (K^2+\xi^2) \bigr\}
},
\end{align}
where
\begin{align*}
M_{\tau,\xi}:=\frac{16}{5} \frac{M+C(2^{k_1} +2^{k_2})}{ |\xi| } + \xi^4.
\end{align*}
We note that there exists $\delta_1>0$ such that 
\begin{align} \label{es_M_2}
( M_{\tau,\xi}^{1/2}-\xi^2 )^{1/2} \geq \delta_{1} |\xi|^{-3} 
(2^{k_1/2}+2^{k_2/2}).
\end{align}
Following (\ref{es_M_1}) and (\ref{es_M_2}), 
the variation of $\xi_1$ is at most 
\begin{align} \label{es_{var_1}} 
O \Bigl( \min \bigl\{ |\xi|^{-1} K^{-3} (2^{k_1}+2^{k_2}),~ |\xi|^{-3/2} (2^{3k_1/4} +2^{3 k_2 /4})   \bigr\} \Bigr).
\end{align}
When 
\begin{align*}
K \leq \biggl[ \Bigl\{ \frac{16}{5} \frac{|M-C(2^{k_1} +2^{k_2})|}{ |\xi| } + \xi^4 \Bigr\}^{1/2}-\xi^2 \biggr]^{1/2},
\end{align*}
the variation of $\xi_1$ is bounded by (\ref{es_{var_1}}) in the same manner as above. Next we also fix $\xi_1$. Then 
\begin{align*}
|\tau_1-\xi_1^5| \lesssim 2^{k_1} \text{ and } 
|(\tau-\tau_1)-(\xi-\xi_1)^5| \lesssim 2^{k_2}
\end{align*}
imply that the variation of $\tau_1$ is at most 
$O \bigl( \min \{2^{k_1},~2^{k_2}  \} \bigr) $. 
Combining this and (\ref{es_{var_1}}), we obtain 
\begin{align*} 
m(\tau,\xi) \lesssim \Bigl\{ |\xi|^{-1} K^{-3} 2^{k_1+k_2},~ |\xi|^{-3/2}
\max\{ 2^{3 k_1/4}, 2^{3 k_2/4} \} \min\{2^{k_1}, 2^{k_2} \} \Bigr\},
\end{align*}
which shows (\ref{m_es}). 
\end{proof}

\begin{lem} \label{lem_dy_2}
Assume that $g$ is supported on a single $A_j$ for $j \geq 0$ and $f$ is an arbitrary test function. Then 
\begin{align} \label{es_dy_2-1}
\| f* g \|_{L_{\tau,\xi}^2( B_k)}
\lesssim 2^{3k/8}~  
\| |\xi|^{-3/4} f \|_{L_{\tau,\xi}^2} 
\| g \|_{\hat{X}_{(2,1)}^{0,1/2}}.
\end{align}
Moreover if a non-empty set $\Omega \subset \mathbb{R}^2$ satisfies 
\begin{align*}
K_1:=\inf \{|\xi+\xi_2|~;~\exists \tau,\tau_2 \text{ s.t. } 
(\tau,\xi) \in \Omega,
~(\tau_2,\xi_2) \in \text{\upshape supp}~g  \}>0,
\end{align*}
then we have
\begin{align} \label{es_dy_2-2}
\| f* g \|_{L_{\tau,\xi}^2(\Omega \cap B_k)}
\lesssim  2^{k/2}~K_1^{-3/2}~  
\| |\xi|^{-1/2} f \|_{L_{\tau,\xi}^2} \| g \|_{\hat{X}_{(2,1)}^{0,1/2} }.
\end{align}
\end{lem}

\begin{proof}
If $g$ is restricted to $B_{k_2}$ for $k_2 \geq 0$, it suffices to show 
\begin{align} \label{red_dy_1}
& \Bigl| \int_{\mathbb{R}^2 } \int_{ \mathbb{R}^2 } f(\tau_1,\xi_1) g(\tau-\tau_1,\xi-\xi_1) h(\tau,\xi) d \tau d\xi d\tau_1 d\xi_1 \Bigr| \nonumber \\
& \hspace{1.2cm} \lesssim 2^{3k/8} 2^{k_2/2} \| |\xi|^{-3/4} f \|_{L_{\tau,\xi}^2} \| g \|_{L_{\tau,\xi}^2} \| h \|_{L_{\tau,\xi}^2}
\end{align}
for $h \in L_{\tau,\xi}^2 (B_k)$ and 
\begin{align} \label{red_dy_2}
& \Bigl| \int_{\mathbb{R}^2 } \int_{ \mathbb{R}^2 } f(\tau_1,\xi_1) g(\tau-\tau_1,\xi-\xi_1) h(\tau,\xi) d \tau d\xi d\tau_1 d\xi_1 \Bigr| \nonumber \\
& \hspace{1.8cm} \lesssim K_1^{-3/2}~ 2^{k/2+ k_2/2} \| |\xi|^{-1/2} f \|_{L_{\tau,\xi}^2} \| g \|_{L_{\tau,\xi}^2} \| h \|_{L_{\tau,\xi}^2}
\end{align}
for $h \in L_{\tau,\xi}^2 (B_k \cap \Omega)$. That is the reason why we use (\ref{red_dy_1}) and the triangle inequality to have 
\begin{align*}
& \Bigl| \int_{\mathbb{R}^2} \int_{\mathbb{R}^2} f(\tau_1,\xi_1) g (\tau-\tau_1,\xi-\xi_1) h(\tau,\xi) d\tau d\xi d\tau_1 d\xi_1 \Bigr| \\
\lesssim & 
\sum_{k_2} \Bigl| \int_{\mathbb{R}^2 } \int_{ \mathbb{R}^2 } f(\tau_1,\xi_1) 
( \chi_{B_{k_2}} g) (\tau-\tau_1,\xi-\xi_1) h(\tau,\xi) d \tau d\xi d\tau_1 d\xi_1 \Bigr| \\
\lesssim &  2^{3k/8} \| |\xi|^{-3/4} f\|_{L_{\tau,\xi}^2}  
\sum_{k_2 } 
2^{k_2/2} \| g \|_{L_{\tau,\xi}^2(B_{k_2})} \| h \|_{L_{\tau,\xi}^2},
\end{align*}
which implies (\ref{es_dy_2-1}). 
Moreover, if we assume (\ref{red_dy_2}), we use the triangle inequality to 
obtain (\ref{es_dy_2-2}) in the same manner as above.

We prove (\ref{red_dy_1}) and (\ref{red_dy_2}). We use Schwarz's inequality twice and Fubini's theorem to have 
\begin{align*}
& \Bigl| \int_{\mathbb{R}^2 } \int_{ \mathbb{R}^2 } f(\tau_1,\xi_1) g(\tau-\tau_1,\xi-\xi_1) h(\tau,\xi) d \tau d\xi d\tau_1 d\xi_1 \Bigr| \\
& \hspace{1.2cm} 
\lesssim \sup_{(\tau_1,\xi_1 ) \in \mathbb{R}^2} m_1(\tau_1,\xi_1)^{1/2} 
\| f \|_{L_{\tau,\xi}^2} \| g \|_{L_{\tau,\xi}^2} \| h \|_{L_{\tau,\xi}^2},
\end{align*}
where 
\begin{align*}
m_1(\tau_1,\xi_1):=\int_{\mathbb{R}^2} \chi_{\Lambda_2} (\tau,\xi,\tau_1,\xi_1) d\tau d\xi
\end{align*}
and 
\begin{align*}
\Lambda_2:=\bigl\{ (\tau,\xi,\tau_1,\xi_1) \in \mathbb{R}^4~;~ (\tau-\tau_1,\xi-\xi_1) \in \text{\upshape supp } f, ~~
(\tau,\xi) \in \text{\upshape supp } h  \bigr\}.
\end{align*}
Therefore (\ref{red_dy_1}) and (\ref{red_dy_2}) are reduced to the estimate. 
\begin{align} \label{es_{m_1}} 
m_1(\tau_1,\xi_1) \lesssim \min \Bigl\{ K_1^{-3} |\xi_1|^{-1} 2^{k+k_2},~
|\xi_1|^{-3/2} 2^{3k/4} 2^{k_2} \Bigr\}.
\end{align}
Now we fix $\tau_1$ and $\xi_1 \neq 0$ and estimate $m_1$. We use the identity 
\begin{align*}
(\tau_1- \frac{\xi_1^5}{16}) -(\tau-\xi^5)+ \bigl\{ (\tau-\tau_1)-(\xi-\xi_1)^5  \bigr\} = 
\frac{5}{16} \xi_1 (2\xi-\xi_1)^2 \bigl\{ (2\xi-\xi_1)^2+ 2\xi_1^2 \bigr\}
\end{align*}
to have
\begin{align*}
& \max \biggl\{ \Bigl\{ \frac{16}{5} \frac{ |M_1- C(2^{k}+ 2^{k_2})|   }{|\xi_1|} + \xi_1^4 \Bigr\}^{1/2}-\xi_1^2,~ K_1^2 \biggr\} \\
& \hspace{1.2cm} \leq |2\xi-\xi_1|^2 
 \leq  \Bigl\{ \frac{16}{5} \frac{ M_1+ C(2^{k}+ 2^{k_2}) }{|\xi_1|} + \xi_1^4 
\Bigr\}^{1/2}-\xi_1^2,
\end{align*}
where $M_1:=|\tau_1-\xi_1^5/16|$. This estimate shows (\ref{es_{m_1}}) by following the proof of Lemma~\ref{lem_dy_1}.
\end{proof}

\begin{lem} \label{lem_dy_3}
Assume that $f$ is supported on a single $A_{j}$ for $j \geq 0$ and $g$ is an arbitrary test function. Then 
\begin{align} \label{es_dy_3-1}
\| f* g \|_{L_{\tau,\xi}^2( B_k)}
\lesssim 2^{3k/8}~  
\| f \|_{\hat{X}_{(2,1)}^{0,1/2}} \| |\xi|^{-3/4}~ g \|_{L_{\tau,\xi}^2}.
\end{align}
Moreover if a non-empty set $\Omega \subset \mathbb{R}^2$ satisfies 
\begin{align*}
K_2:=\inf \{|\xi+\xi_1|~;~\exists \tau,\tau_1 \text{ s.t. } (\tau,\xi) \in \Omega,~(\tau_1,\xi_1) \in \text{\upshape supp}~f  \}>0,
\end{align*}
then we have
\begin{align} \label{es_dy_3-2}
\| f* g \|_{L_{\tau,\xi}^2(\Omega \cap B_k)}
\lesssim  2^{k/2}~K_2^{-3/2}~  
\| f \|_{\hat{X}_{(2,1)}^{0,1/2} } \| |\xi|^{-1/2}~g  \|_{L_{\tau,\xi}^2} .
\end{align}
\end{lem}

In the same manner as the proof of Lemma~\ref{lem_dy_2}, we immediately obtain
(\ref{es_dy_3-1}) and (\ref{es_dy_3-2}) by symmetry. 
We put a smooth cut-off function $\varphi (t)$ satisfying 
$\varphi (t)= 1~\text{for}~|t|<1
~\text{and}~   =0~\text{for}~  |t|>2 $ and define $\| \cdot \|_{Z^{s,a}}$ as 
$\| u \|_{Z^{s,a}}:=\| \widehat{u} \|_{\hat{Z}^{s,a}}$. 
We mention the linear estimates below. 
\begin{prop} \label{prop_linear1}
Let $s,a \in \mathbb{R}$ and $u(t)=\varphi (t) U(t) u_0$. Then the following estimate holds. 
\begin{align*}
\| u\|_{Z^{s, a}} +\| u \|_{L_t^{\infty}( \mathbb{R};H_{x}^{s,a})} \lesssim
\| u_0 \|_{H^{s,a}}.
\end{align*}
\end{prop}
\begin{prop}\label{prop_linear2}
Let $s,a \in \mathbb{R}$ and 
\begin{align*}
u(t)=\varphi (t) \int_{0}^{t} U(t-s) F(s) ds .
\end{align*}
Then the following estimate holds.
\begin{align*}
\| u \|_{Z^{s, a}}+ \| u \|_{L_t^{\infty} (\mathbb{R};H_x^{s,a})}
\lesssim \|\mathcal{F}_{\tau,\xi}^{-1} \langle \tau-\xi^5 \rangle^{-1} 
\widehat{F}  \|_{Z^{s,a}}+ 
\| \langle \xi \rangle^{s-a} ~|\xi|^{a}~  \langle \tau-\xi^5 \rangle^{-1} 
\widehat{F}  \|_{L_{\xi}^2 L_{\tau}^1}.
\end{align*} 
\end{prop}
The proofs of these propositions are given 
in \cite{GTV}.  

\section{Proof of the bilinear estimates}

\noindent
In this section, we prove the bilinear estimate (\ref{BE-2}). We use the following notation for simplicity, 
\begin{align*}
A_{<j_1} := \bigcup_{j < j_1} A_j, \hspace{0.3cm} 
B_{[k_1,k_2)}:= \bigcup_{k_1 \leq k <k_2} B_k,
\hspace{0.3cm} \text{etc}.
\end{align*}
Here we state the key bilinear estimates as follows. 
\begin{prop} \label{prop_BE-2}
Let $s,a$ satisfy (\ref{co_op}). Suppose that $f$ and $g$ are restricted on 
$A_{j_1}$ and $A_{j_2}$ for $j_1,j_2 \in \mathbb{N} \cup \{ 0 \}$. Then we obtain, for $j \geq 0$,
\begin{align} \label{BE-X}
& \| \langle \tau-\xi^5 \rangle^{-1}~\xi~(\xi^2 f)*g \|_{\hat{Z}^{s,a}(A_j)}
 \lesssim C(j,j_1,j_2) \| f \|_{\hat{Z}^{s,a}} \| g \|_{\hat{Z}^{s,a}}, \\
\label{BE-Y}
& \bigl\|~\langle \xi \rangle^{s-a}~
|\xi|^{a+1}~ \langle \tau-\xi^5 \rangle^{-1}~(\xi^2 f)*g 
 \bigr\|_{L_{\xi}^2 L_{\tau}^1 (A_j)} 
\lesssim C(j,j_1,j_2) \| f \|_{\hat{Z}^{s,a}} \| g\|_{\hat{Z}^{s,a}},
\end{align}
in the following eight cases. 

(i) At least two of $j,j_1,j_2$ are less than $30$ and $C(j,j_1,j_2) \sim 1$.

(ii) $j_1,j_2 \geq 30$, $|j_1-j_2| \leq 1$, $0<j <j_1-9$ and $C(j,j_1,j_2) \sim 2^{-\delta j}$ for some $\delta>0$.

(iii) $j,j_2 \geq 30$, $|j-j_2| \leq 10$, $0< j_1 < j-10$ and $C(j,j_1,j_2) \sim2^{-\delta j_1}+2^{-\delta (j-j_1) }$ for some $\delta>0$.

(iv) $j,j_1 \geq 30$, $ |j-j_1| \leq 10$, $0< j_2 < j-10$ and $C(j,j_1,j_2) \sim 2^{-\delta j_2}+2^{-\delta (j-j_2)}$ for some $\delta>0$.

(v) $j,j_1,j_2 \geq 30$, $|j-j_1| \leq 10$, $|j-j_2| \leq 10$ and 
$C(j,j_1,j_2) \sim 1$.

(vi) $j_1,j_2 \geq 30$, $j=0$ and $C(j,j_1j_2) \sim 1$.

(vii) $j,j_2 \geq 30$, $j_1=0$ and $C(j,j_1,j_2) \sim 1$.

(viii) $j, j_1 \geq 30$, $j_2=0$ and $C(j,j_1,j_2) \sim 1$.
\end{prop}

Combining the $L_{\xi}^2$-property of $\hat{Z}^{s,a}$, namely $\| f \|_{\hat{Z}^{s,a}}^2=\sum_{j} \| f \|_{\hat{Z}^{s,a}(A_j)}^2 $, and this proposition, we obtain (\ref{BE-2}). 

\begin{proof}
We put $2^{k_{\max}}:=\max\{2^k, 2^{k_1}, 2^{k_2}  \}$. Then we have 
\begin{align*}
2^{k_{\max}} \gtrsim
\bigl| \xi \xi_1 (\xi-\xi_1) \bigl\{ \xi^2+ \xi_1^2+(\xi-\xi_1)^2 
\bigr\} \bigr|. 
\end{align*}
From the definition, we easily obtain 
\begin{align} \label{imb}
\hat{X}^{s,a,3/4+\varepsilon} \hookrightarrow \hat{Z}^{s,a} 
\hookrightarrow \hat{X}^{s,a,3/8}.
\end{align}
where $\varepsilon>0$ is sufficiently small.

\vspace{0.3em}

(I) Estimate for (i). In this case, we can assume $j,j_1,j_2  \leq 40$. The left hand sides of (\ref{BE-X}) and 
(\ref{BE-Y}) is bounded by $C \| |\xi|^{a+1} 
\langle \tau-\xi^5 \rangle^{-1/4+\varepsilon} f*g  \|_{L_{\tau,\xi}^2}$ from 
(\ref{imb}). 
We use the H\"{o}lder inequality and the Young inequality to obtain
\begin{align*}
 \||\xi|^{a+1} \langle \tau-\xi^5 
\rangle^{-1/4+\varepsilon}  f*g  \|_{L_{\tau,\xi}^2} & \lesssim 
\| f*g \|_{L_{\xi}^{\infty} L_{\tau}^2 } \\
& \lesssim \| f \|_{L_{\xi}^2 L_{\tau}^{4/3}}\| g \|_{L_{\xi}^2 L_{\tau}^{4/3}}
\lesssim \| f   \|_{\hat{X}^{0,a,3/8}} \| g \|_{\hat{X}^{0,a,3/8}},
\end{align*}
which implies the desired estimate from (\ref{imb}). 

From the estimate in the cases (iv) and (viii), we easily obtain 
(\ref{BE-X}) in the cases (iii) and (vii) because 
we recover derivative losses in these cases. Therefore we omit the proof in the cases (iii) and (vii).
We first prove (\ref{BE-X}) in other cases.

\vspace{0.3em}

(II) Estimate for (ii). We prove 
\begin{align} \label{es-HHL}
2^{(s+1) j}~2^{2j_1} \sum_{k \geq 0} 2^{-k/2} \| f*g \|
_{L_{\tau,\xi}^{2}(A_j \cap B_k)}
\lesssim 2^{-\delta j} \| f \|_{\hat{X}_{(2,1)}^{s,1/2}} \| g \|_{\hat{X}_{(2,1)}^{s,1/2}}.
\end{align}

(IIa) We consider (\ref{es-HHL}) in the case $2^{k_{\max}}=2^{k}$. From 
$2^{k} \gtrsim 2^{4j_1+j}$, we use (\ref{es_dy_1-2}) with $K \sim 2^{j_1}$ to have
\begin{align*}
\text{(L.H.S.)} & 
\sim 2^{(s+1)j}~ 2^{ (-2s+2) j_1} \sum_{k \geq 4j_1 +j +O(1)} 2^{-k/2}
\| (\langle \xi \rangle^s f)*( \langle \xi \rangle^s g)  
\|_{L_{\tau,\xi}^2 (B_k)} \\
& \lesssim 2^{(s+1/2) j}~ 2^{-2s j_1} 
\|( \langle \xi \rangle^s f) *(\langle \xi \rangle^s g )  \|_{L_{\tau,\xi}^2} \\&  \lesssim 2^{s j}~ 2^{(-2s-3/2) j_1} 
\| f \|_{\hat{X}_{(2,1)}^{s,1/2}} \| g \|_{\hat{X}_{(2,1)}^{s,1/2}}
\end{align*}
which is bounded by $2^{-5j/4} \| f \|_{\hat{X}_{(2,1)}^{s,1/2}} \| g \|_{\hat{X}_{(2,1)}^{s,1/2}}$ for $s \geq -1/4$.

\vspace{0.3em}

(IIb) We consider (\ref{es-HHL}) in the case $2^{k_{\max}}=2^{k_2}$. From 
$2^{-k_2/2} \lesssim 2^{-k/8} 2^{-3 j_1/2} 2^{-3j/8}$, we use (\ref{es_dy_3-2}) with $K_2 \sim 2^{j_1}$ to obtain
\begin{align*}
\text{(L.H.S.)} & \lesssim 2^{(s+5/8) j}~ 2^{(-2s+1/2) j_1} 
\sum_{k \geq 0} 2^{-5k/8} 
\| (\langle \xi \rangle^s f) * (\langle \xi \rangle^s 
\langle \tau-\xi^5 \rangle^{1/2} g) \|_{L_{\tau,\xi}^2(B_k)} \\
& \lesssim 2^{(s+5/8) j}~ 2^{(-2s-3/2) j_1} \sum_{k \geq 0} 2^{-k/8}
\| f \|_{\hat{X}_{(2,1)}^{s,1/2}} \| g \|_{\hat{X}_{(2,1)}^{s,1/2}},
\end{align*}
which shows the required estimate for $s \geq -1/4$. 

\vspace{0.3em}

In the same manner as above, we obtain the desired estimate in the case $2^{k_{\max}}=2^{k_1}$ by symmetry. 

\vspace{0.5em}

(III) Estimate for (iii). We prove 
\begin{align} \label{es-HLH}
2^{(s+3)j} \sum_{k \geq 0} 2^{-k/2} \| f*g \|_{L_{\tau,\xi}^2(B_k)}
\lesssim \bigl(2^{-\delta j_1}+2^{-\delta(j-j_1)} \bigr)
\|  f \|_{\hat{X}_{(2,1)}^{s,1/2}} \|  g \|_{\hat{X}_{(2,1)}^{s,1/2}}.
\end{align}

(IIIa) We consider (\ref{es-HLH}) in the case $2^{k_{\max}}=2^{k}$. 
Since $2^{k} \gtrsim 2^{4j+j_2}$, we use (\ref{es_dy_1-2}) 
with $K \sim 2^{j}$ to have 
\begin{align*}
\text{(L.H.S.)} \sim & 2^{-sj_2} ~2^{3j} \sum_{k \geq 4j+j_2+O(1)} 
2^{-k/2} \| (\langle \xi \rangle^s f)*( \langle \xi \rangle^s g) 
\|_{L_{\tau,\xi}^2(B_k )} \\
\lesssim & 2^{(-s- 1/2) j_2}~ 2^{j} \| (\langle \xi \rangle^s f)*
(\langle \xi \rangle^s g) \|_{L_{\tau,\xi}^2 } \\
\lesssim & 2^{(-s-1/2) j_2}~ 2^{-j} \| f \|_{\hat{X}_{(2,1)}^{s,1/2}}
\| g  \|_{\hat{X}_{(2,1)}^{s,1/2}}.
\end{align*}

(IIIb) We consider (\ref{es-HLH}) in the case $2^{k_{\max}}=2^{k_1}$. 
From $2^{k_1} \gtrsim 2^{4j+j_2}$, we have $2^{-k_1/2} \lesssim 
2^{-k/8} 2^{-3j/2} 2^{-3 j_2/8} $. Then we use (\ref{es_dy_2-2}) with 
$K_1 \sim 2^{j}$ to have 
\begin{align*}
\text{(L.H.S.)} \lesssim & 2^{(-s-3/8) j_2}~ 2^{3j /2} \sum_{k \geq 0} 
2^{-5k/8} \| (\langle \xi \rangle^s \langle \tau-\xi^5 \rangle^{1/2}
 f)* (\langle \xi \rangle^s g)  
\|_{L_{\tau,\xi}^2(B_k)} \\
\lesssim & 2^{(-s-3/8) j_2}~ 2^{-j/2} \sum_{k \geq 0} 2^{-k/8} 
\| f \|_{\hat{X}_{(2,1)}^{s,1/2}} \|  g \|_{\hat{X}_{(2,1)}^{s,1/2}},
\end{align*}
which implies the desired estimate for $s \geq -1/4$.

\vspace{0.3em}

(IIIc) We consider (\ref{es-HLH}) in the case $2^{k_{\max}}=2^{k_2}$. 
Since $2^{-k_2/2} \lesssim 2^{-k/8} 2^{-3 j /2} 2^{-3j_2 / 8}  $, we use 
(\ref{es_dy_3-2}) with $K_2 \sim 2^{j}$ to have 
\begin{align*}
\text{(L.H.S.)} \lesssim & 2^{(-s-3/8) j_2} 2^{3j/2} \sum_{k \geq 0}
2^{-5k/8} \|  (\langle \xi \rangle^s f) * (\langle \xi \rangle^s g ) 
\|_{L_{\tau,\xi}^2(B_k)} \\
\lesssim & 2^{(-s-7/8) j_2} \sum_{k \geq 0} 2^{-k/8} 
\| f \|_{\hat{X}_{(2,1)}^{s,1/2}}  \| g \|_{\hat{X}_{(2,1)}^{s,1/2}}, 
\end{align*}
which shows the required estimate. 

(VI) Estimate for (v). We prove 
\begin{align} \label{es_HHH}
2^{(s+3)j} \sum_{k \geq 0} 2^{-k/2} \| f*g \|_{L_{\tau,\xi}^2 (B_k)} \lesssim
\| f \|_{\hat{X}_{(2,1)}^{s,1/2}} \| g \|_{\hat{X}_{(2,1)}^{s,1/2}}.
\end{align}

(VIa) We consider (\ref{es_HHH}) in the case $2^{k_{\max}}=2^{k}$. 
Since $2^{k} \gtrsim 2^{5j}$, we have 
\begin{align*}
\text{(L.H.S.)} \sim &2^{(-s+3) j} \sum_{k \geq 5j+O(1)} 2^{-k/2} \| (\langle 
\xi \rangle^s f) *(\langle \xi \rangle^s g) \|_{L_{\tau,\xi}^2 (B_k)} \\
\lesssim & 2^{(-s+1/2)j } \|( \langle \xi \rangle^s f
)* (\langle \xi \rangle^s g)  \|_{L_{\tau,\xi}^2 },
\end{align*}
which shows the desired estimate by using (\ref{es_dy_1-1}).

\vspace{0.3em}

(VIb) We consider (\ref{es_HHH}) in the case $2^{k_{\max}}=2^{k_1}$. 
Since $2^{k_1} \gtrsim 2^{5j_1}$, we use (\ref{es_dy_2-1}) with $K_1 \sim 2^{j_1}$ to have 
\begin{align*}
\text{(L.H.S.)} \sim & 2^{(-s+3)j_1} \sum_{k \geq 0} 2^{-k/2} 
\| (\langle \xi \rangle^s f)* (\langle \xi \rangle^s g) \|
_{L_{\tau,\xi}(B_k)} \\
\lesssim & 2^{(-s+1/2)j_1} \sum_{k \geq 0} 2^{-k/2} 
\| (\langle \xi \rangle^s \langle \tau-\xi^5 \rangle^{1/2} f) * (\langle \xi \rangle^s g) 
\|_{L_{\tau,\xi}^2} \\
\lesssim & 2^{(-s-1/4)j_1} \sum_{k \geq 0} 2^{-k/8}
\| f \|_{\hat{X}_{(2,1)}^{s,1/2}} \| g \|_{\hat{X}_{(2,1)}^{s,1/2}}.
\end{align*}

In the same manner as above, we obtain the desired estimate in the case 
$2^{k_{\max}}=2^{k_2}$ by symmetry.

\vspace{0.5em}

(V) Estimate of (v). We prove 
\begin{align} \label{es_hhl-1}
2^{2j_1} \| \langle \tau-\xi^5 \rangle^{-1} \xi f*g   \|_{\hat{X}_L^a}
\lesssim \| f \|_{\hat{X}_{(2,1)}^{s,1/2}} \|  g \|_{\hat{X}_{(2,1)}^{s,1/2}}.
\end{align}
We remark that 
\begin{align} \label{imb_a}
\| f \|_{\hat{X}_L^{a,3/8}} \leq \| f \|_{\hat{X}_L^a} 
\leq \| f \|_{\hat{X}_L^{a,3/4,1}}.
\end{align}
In the case $|\xi| \leq 2^{-4j_1}$, from (\ref{imb_a}), it suffices show to 
\begin{align*}
2^{2j_1} \| |\xi|^{a+1} \langle \tau \rangle^{-1/4+\varepsilon} 
f* g \|_{L_{\tau,\xi}^2} \lesssim \| f \|_{\hat{X}_{(2,1)}^{s,1/2}}
\|g \|_{\hat{X}_{(2,1)}^{s,1/2}}.
\end{align*}
We use the H\"{o}lder inequality and Young inequality to have 
\begin{align*}
\text{(L.H.S.)} \lesssim & 2^{(-2s+2)j_1} \| |\xi|^{a+1} 
\langle \tau \rangle^{-1/4+\varepsilon} ( \langle \xi \rangle^s f )* 
(\langle \xi \rangle^s g) \|_{L_{\tau,\xi}^2} \\
\lesssim & 2^{(-2s+2)j_1} \| |\xi|^{a+1} \|_
{L_{\xi}^2 (|\xi| \leq 2^{-4 j_1})} \| (\langle \xi \rangle^s f) *
(\langle \xi \rangle^s g) \|_{L_{\xi}^{\infty} L_{\tau}^2 } \\
\lesssim & 2^{-2(s+2a+2) j_1 } 
\| \langle \xi \rangle^s f\|_{L_{\xi}^2 L_{\tau}^{4/3}}
\| \langle \xi \rangle^s g \|_{L_{\xi}^2 L_{\tau}^{4/3}}, 
\end{align*}
which implies the required estimate since $\| f \|_{L_{\xi}^2 L_{\tau}^p } \lesssim \| f \|_{\hat{X}_{(2,1)}^{0,1/2}} $ when $1 \leq p \leq 2$. 
Therefore we only consider the case $2^{-4j_1} \leq |\xi| \leq 1$. 

\vspace{0.3em}

(Va) We consider (\ref{es_hhl-1}) in the case $2^{k_{\max}}=2^{k_2}$. 
Note that the left hand side of (\ref{es_hhl-1}) is bounded by 
\begin{align} \label{Va}
 2^{(-2s+2) j_1} \sum_{k \geq 0} 2^{-k/4} 
\| |\xi|^{a+1} (\langle \xi \rangle^s f) *( \langle \xi \rangle^s g) 
\|_{L_{\tau,\xi}^2 (B_k)}
\end{align}
Since $-s/2 \leq a+1$ and $2^{k_2} \gtrsim |\xi| 2^{4j_1}$, we have 
\begin{align*}
2^{-2sj_1} |\xi|^{a+1} \lesssim (|\xi| 2^{4 j_1})^{-s/2} \lesssim 2^{k_{2}/8}
\lesssim 2^{k_2/2} 2^{-3k/8}. 
\end{align*}
Then we use (\ref{es_dy_3-2}) with $K_2 \sim 2^{j_1}$ to obtain 
\begin{align*}
\text{(\ref{Va})} \lesssim &
2^{2 j_1} \sum_{k \geq 0} 2^{-5k/8} 
\| (\langle \xi \rangle^s f) * (\langle \xi \rangle^s \langle \tau-\xi^5 
\rangle^{1/2} g) \|_{L_{\tau,\xi}^2 (B_k)} \\
\lesssim & \sum_{k \geq 0} 2^{-k/8} \| f \|_{\hat{X}_{(2,1)}^{s,1/2}}
\| g  \|_{\hat{X}_{(2,1)}^{s,1/2}}. 
\end{align*}

In the same manner as above, we obtain the desired estimate in the case 
$2^{k_{\max}}=2^{k_2}$ by symmetry. 

\vspace{0.5em}

(Vb) We consider (\ref{es_hhl-1}) in the case $2^{k_{\max}}=2^k$. 
If $2^k \gg |\xi| 2^{4j_1}$, then we have $2^{k_{\max}} \sim 2^{k_1}$ or 
$2^{k_2}$. Thus we only prove (\ref{es_hhl-1}) in the case $2^{k_{\max}} \sim 
|\xi| 2^{4j_1}$. 

(Vb-1) Firstly, we prove (\ref{es_hhl-1}) in the case $-7/8<a<-1/4$.

(i) We first consider (\ref{es_hhl-1}) when $f*g$ is restricted to $D_1$. 
In this case, we have $2^{-3j_1/2 } \lesssim |\xi| \leq 1$ and 
$2^{5j_1/2 } \lesssim |\tau| \lesssim 2^{4j_1}$. 

(ia) In the case $a=-1/4$, we prove
\begin{align} \label{es-hhl-1}
2^{2j_1} \| |\xi|^{3/4} \langle \tau \rangle^{-1/4} f*g \|_{L_{\tau,\xi}^2}
\lesssim \| f \|_{\hat{X}_{(2,1)}^{s,1/2}} \| g \|_{\hat{X}_{(2,1)}^{s,1/2}}.
\end{align}
Since $|\tau| \sim |\xi| 2^{4j_1}$, we use (\ref{es_dy_1-2}) with 
$K \sim 2^{j_1}$ to have
\begin{align*}
\text{(L.H.S.)} \sim & 2^{-2sj_1 +j_1} \| |\xi|^{1/2} 
(\langle \xi \rangle^s f)*(\langle \xi \rangle^s g) \|_{L_{\tau,\xi}^2} \\
\lesssim & 2^{-2(s+1/4)j_1} \| f \|_{\hat{X}_{(2,1)}^{s,1/2}} 
\| g \|_{\hat{X}_{(2,1)}^{s,1/2}}.
\end{align*}

(ib) In the case $-7/8 \leq a <1/4$, we prove 
\begin{align} \label{es-hhl-2}
2^{2j_1} \sum_{k \geq 5j_1/2+O(1)} 2^{(3a/5-1/10) k} \| |\xi|^{a+1} 
f*g \|_{L_{\tau,\xi}^2 (B_k)} \lesssim 
\| f \|_{\hat{X}_{(2,1)}^{s,1/2}} \| g  \|_{\hat{X}_{(2,1)}^{s,1/2}}. 
\end{align}
Since $|\xi|^{a+1/2} \sim 2^{(a+1/2) k} 2^{-4aj_1-2j_1}$, we use 
(\ref{es_dy_1-2}) with $K \sim 2^{j_1} $ to obtain
\begin{align*}
\text{(L.H.S.)} \sim &  2^{-2sj_1-4aj_1} \sum_{k \geq 5j_1/2+O(1)} 
2^{\frac{8}{5}(a+\frac{1}{4}) k} \| |\xi|^{1/2} (\langle \xi \rangle^s f)*
(\langle \xi \rangle^s g)  \|_{L_{\tau,\xi}^2(B_k) } \\
\lesssim & 2^{-2sj_1+j_1}  \| |\xi|^{1/2} (\langle \xi \rangle^s f)*(\langle \xi \rangle^s g)  \|_{L_{\tau,\xi}^2} \\
\lesssim & 2^{-2(s+1/4) j_1} \| f \|_{\hat{X}_{(2,1)}^{s,1/2}} 
\| g \|_{\hat{X}_{(2,1)}^{s,1/2}}.
\end{align*}

(ii) We next consider (\ref{es_hhl-1}) when $f*g$ is restricted to $D_2$. In the present case, we have $2^{-4j_1} \leq |\xi| \lesssim 2^{-3j_1/2}$ and 
$1 \lesssim |\tau| \lesssim 2^{5j_1/2}$. 

In the case $-7/8<a \leq -1/4$, we prove 
\begin{align} \label{es_hhl-3}
2^{2j_1} \sum_{k \leq 5j_1/2+O(1)} 2^{(3a/5-1/10) k} 
\| |\xi|^{a+1} f*g \|_{L_{\tau,\xi}^2 (B_k)} \lesssim 
\| f \|_{\hat{X}_{(2,1)}^{s,1/2}} \| g \|_{\hat{X}_{(2,1)}^{s,1/2}}.
\end{align}
Since $|\xi| \sim 2^{k-4j_1}$, we use the H\"{o}lder inequality and the 
Young inequality to have
\begin{align*}
\text{(L.H.S.)} \sim & 2^{-2sj_1+2j_1} \sum_{k \leq 5j_1/2+O(1)} 
2^{(3a/5-1/10) k} \| |\xi|^{a+1} (\langle \xi \rangle^s f) * 
(\langle \xi \rangle^s g ) \|_{L_{\tau,\xi}^2 (B_k)} \\
\lesssim & 2^{-2sj_1+2j_1} \sum_{k \leq 5j_1/2+O(1)} 2^{(3a/5-1/10)k} \\
& \hspace{1.8cm} \times \| |\xi|^{a+1} \|_{L_{\xi}^2(|\xi| \sim 2^{k-4j_1})} 
\| (\langle \xi \rangle^s f )*( \langle \xi \rangle^s g) 
\|_{L_{\xi}^{\infty} L_{\tau}^2} \\
\sim & 2^{-2sj_1 -4a j_1-4j_1} \sum_{k \leq 5j_1/2+O(1)} 
2^{\frac{8}{5}(a+\frac{7}{8})k } 
\| \langle \xi \rangle^s f \|_{L_{\xi}^2 L_{\tau}^1} 
\| \langle \xi \rangle^s g \|_{L_{\tau,\xi}^2} \\
\lesssim & 2^{-2(s+1/4) j_1} \| f \|_{\hat{X}_{(2,1)}^{s,1/2}} 
\| g \|_{\hat{X}_{(2,1)}^{s,1/2}}. 
\end{align*}

(Vb-2) Secondly, when $a=-7/8$, we prove 
\begin{align*} 
2^{2j_1} \| |\xi|^{1/8} \langle \tau \rangle^{-5/8+\varepsilon_1/2} f*g  
\|_{L_{\tau,\xi}^2} \lesssim \| f \|_{\hat{X}_{(2,1)}^{s,1/2}} 
\| g \|_{\hat{X}_{(2,1)}^{s,1/2}}. 
\end{align*}
Since $2^{k} \sim |\xi| 2^{4j_1}$ and $s \geq -1/4+\varepsilon_2$, 
we use the H\"{o}lder inequality and the Young inequality to obtain 
\begin{align*}
\text{(L.H.S.)} \sim & 2^{(-2s-1/2+2\varepsilon_1)} \| |\xi|^{-1/2+\varepsilon_1/2} (\langle \xi \rangle^s f) * (\langle \xi \rangle^s g) 
\|_{L_{\tau,\xi}^2} \\
\lesssim & \| |\xi|^{-1/2+\varepsilon_1/2} \|_{L_{\xi}^2 (|\xi| \leq 1)} 
\| (\langle \xi \rangle^s f) * (\langle \xi \rangle^s g) \|_{L_{\xi}^{\infty} 
L_{\tau}^2} \\
\lesssim & \| f \|_{\hat{X}_{(2,1)}^{s,1/2}} \| g \|_{\hat{X}_{(2,1)}^{s,1/2}}.
\end{align*}

(Vb-3) Finally, when $-3/2 <a< -7/8$, we prove 
\begin{align} \label{es_hhl-2}
2^{2j_1} \| |\xi|^{a+1} \langle \tau \rangle^{-5/8+\varepsilon_1/2} f*g  \|_{L_{\tau,\xi}^2 }
\lesssim \| f \|_{\hat{X}_{(2,1)}^{s,1/2}} \| g  \|_{\hat{X}_{(2,1)}^{s,1/2}}.
\end{align}
Since $|\xi|^{a+1} \leq |\xi|^{-s/2}$ and $s \geq -1/4+2\varepsilon_1$, we have
\begin{align*}
|\xi|^{a+1} \langle \tau \rangle^{-5/8+ \varepsilon_1/2} \lesssim 
|\xi|^{-1/2-\varepsilon_1/2} 2^{(2s-2-2\varepsilon_1 )j_1}. 
\end{align*}
From this, we use the H\"{o}lder inequality and the Young inequality to obtain 
\begin{align*}
\text{(L.H.S.)} \lesssim & 2^{-2 \varepsilon_1 j_1} \| \langle \tau \rangle^{-1/2-\varepsilon_1/2} 
(\langle \xi \rangle^s f)* ( \langle  \xi \rangle^s g) \|_{L_{\tau,\xi}^2} \\
\lesssim & 2^{-2 \varepsilon_1 j_1} 
\| |\xi|^{-1/2-\varepsilon_1/2} \|_{L_{\xi}^2 (2^{-4j_1} \leq |\xi| )}
\| \langle \xi \rangle^s f \|_{L_{\xi}^2 L_{\tau}^1} \| 
\langle \xi \rangle^s g   \|_{L_{\tau,\xi}^2} \\
\lesssim 
& \| f \|_{\hat{X}_{(2,1)}^{s,1/2}} \| g \|_{\hat{X}_{(2,1)}^{s,1/2}}. 
\end{align*}

\vspace{0.5em}

(VI) Estimate for (viii). We prove 
\begin{align} \label{es_hlh-1}
2^{3j} \sum_{k \geq 0} 2^{-k/2} \| (\langle \xi \rangle^s f)* g   
\|_{L_{\tau,\xi}^2 (B_k)} \lesssim
 \| f\|_{\hat{X}_{(2,1)}^{s,1/2}} \| g \|_{\hat{X}_L^a}.
\end{align}
In the case $|\xi_2| \leq 2^{-4j}$, we easily obtain the desired estimate for 
$a \leq -1/4$. 
Hence we only consider the case $2^{4j} \leq |\xi_2| \leq 1$. 

\vspace{0.3em}

(VIa) We consider (\ref{es_hlh-1}) in the case $2^{k_{\max}}=2^{k}$. 
From (\ref{imb_a}), it suffices to show that
\begin{align*}
2^{3j} \| \langle \tau-\xi^5 \rangle^{-1/2+\varepsilon} 
(\langle \xi \rangle^s f) * g \|_{L_{\tau,\xi}^2}
\lesssim \| f \|_{\hat{X}_{(2,1)}^{s,1/2}} 
\| g \|_{\hat{X}_L^{a,3/8}}.
\end{align*}
Since 
$2^{(-1/2+\varepsilon)k} \lesssim 
|\xi_2|^{-1/4 } 2^{-j} 2^{(-1/4+\varepsilon)k_2}$, we use (\ref{es_dy_1-2}) with $K \sim 2^{j}$ to have 
\begin{align*}
2^{3j} \| \langle \tau-\xi^5 \rangle^{-1/2+\varepsilon} 
(\langle \xi \rangle^s f)* g \|_{L_{\tau,\xi}^2} 
\lesssim &
2^{2j} \| (\langle \xi \rangle^{s} f )* (|\xi|^{-1/4} 
\langle \tau \rangle^{-1/4+\varepsilon} g )\|_{L_{\tau,\xi}^2} \\
\lesssim & 
\| f \|_{\hat{X}_{(2,1)}^{s,1/2}} \| g \|_{\hat{X}_L^{-1/4,3/8}},
\end{align*}
which implies the desired estimate for $a \leq -1/4$. 

\vspace{0.3em}

(VIb) We consider (\ref{es_hlh-1}) in the case $2^{k_{\max}}=2^{k_1}$. 
Similar to above, it suffices to show 
\begin{align*}
2^{3j } \sum_{k \geq 0} \|(\langle \xi \rangle^s f) * g 
\|_{L_{\tau,\xi}^2(B_k)}  \lesssim 
\| f \|_{\hat{X}_{(2,1)}^{s,1/2}} \| g \|_{\hat{X}_L^{a,3/8}}.
\end{align*}
Since $2^{-k_1/2} \lesssim 2^{- k/12} |\xi_2|^{-1/4} 2^{-j} 2^{- k_2/6} $, we use (\ref{es_dy_2-2}) with $K_1 \sim 2^{j}$ 
to obtain 
\begin{align*}
\text{(L.H.S.)} \lesssim & 2^{2j} \sum_{k \geq 0} 2^{-7 k/12} 
\| (\langle \xi \rangle^s \langle \tau-\xi^5 \rangle^{1/2} f)* 
(|\xi|^{-1/4} \langle \tau \rangle^{-1/6} g) 
\|_{L_{\tau,\xi}^2(B_k)} \\
\lesssim & \sum_{k \geq 0} 2^{- k/12} \| f \|_{\hat{X}_{(2,1)}^{s,1/2}}
\| g \|_{\hat{X}_L^{-1/4,3/8}},
\end{align*}
which shows the required estimate. 

\vspace{0.3em}

(VIc) We consider (\ref{es_hlh-1}) in the case $2^{k_{\max}}=2^{k_2}$. 
If $2^{k_{\max}} \gg |\xi_2| 2^{4j}$, we have $2^{k_{\max}} \sim 2^{k}$ or $2^{k_{\max}} \sim 2^{k_1}$. 
We only prove the case $2^{k_{\max}} \sim |\xi_2| 2^{4j}$. 

(VIc-1) Firstly, we prove the following estimate in the case $a=-7/8$.
\begin{align*}
2^{3j} \sum_{k \geq 0} 2^{-k/2} 
\| (\langle \xi \rangle^s f)*  g \|_{L_{\tau,\xi}^2 (B_k)} \lesssim \| f \|_{\hat{X}_{(2,1)}^{s,1/2}} \| g \|_{\hat{X}_L^
{-7/8,3/8+\varepsilon_1/2} }. 
\end{align*}
From $|\xi_2|^{3/8} \langle \tau_2 \rangle^{-3/8-\varepsilon_1/2} \lesssim 
2^{-\varepsilon_1 k/2} 2^{-3j/2}$, 
we use (\ref{es_dy_3-2}) with $K_2 \sim 2^{j}$ to obtain 
\begin{align*}
\text{(L.H.S.)} \lesssim & 2^{3j/2} \sum_{k \geq 0}~2^{(-1/2-\varepsilon_1/2)k } ~\| (\langle \xi \rangle^s f)* (|\xi|^{-3/8} \langle \tau \rangle^{3/8+\varepsilon_1/2} g) 
\|_{L_{\tau,\xi}^2 (B_k)} \\
\lesssim & \sum_{k \geq 0} 2^{-\varepsilon_1 k/2} \| f \|_{\hat{X}_{(2,1)}^{s,1/2}} \| g \|_{\hat{X}_L^{-7/8,3/8+\varepsilon_1/2}}.
\end{align*}

(VIc-2) Secondly, we prove the following estimate in the case $-3/2<a<-7/8$. 
\begin{align*}
2^{3j} \sum_{k \geq 0} 2^{-k/2} \|( \langle \xi \rangle^s f)* g \|_{L_{\tau,\xi}^2(B_k)} \lesssim \| f \|_{\hat{X}_{(2,1)}^{s,1/2}} \| g \|_{\hat{X}_L^{a,3/8+\varepsilon_2/2}}.
\end{align*}
We use (\ref{es_dy_3-2}) with $K_2 \sim 2^{j}$ to obtain 
\begin{align*}
\text{(L.H.S.)} \lesssim 2^{3j/2} \sum_{k \geq 0}~1~\| f \|_{\hat{X}_{(2,1)}^{s,1/2}} \| |\xi|^{-1/2} g \|_{L_{\tau,\xi}^2}. 
\end{align*}
Following 
\begin{align*}
|\xi_2|^{-a-1/2} \langle \tau_2 \rangle^{-3/8-\varepsilon_2/2}
\lesssim |\xi_2|^{-a-7/8}~ 2^{-3j/2}~ 2^{-\varepsilon_2 k/2} 
\lesssim 2^{-3j/2}~ 2^{-\varepsilon_2 k/2}, 
\end{align*}
the right hand side is bounded by 
$C \displaystyle \sum_{k \geq 0} 2^{-\varepsilon_2 k/2} \| f \|_{\hat{X}_{(2,1)}^{s,1/2}} \| g \|_{\hat{X}_L^{a,3/8+\varepsilon_2/2}}$.

(VIc-3) Finally, we prove (\ref{es_hlh-1}) in the case 
$-7/8< a< -1/4$.
We consider (\ref{es_hlh-1}) when $g$ is restricted to $D_2$. In the present case, we have 
$ 2^{-4j} \leq |\xi_2| \lesssim 2^{-3j/2} $ and $ 1 \lesssim |\tau_2| \lesssim 2^{5j/2}  $. 

(ia) In the case $a=-1/4$, we prove 
\begin{align} \label{es_hlh-2}
2^{3j} \sum_{k \geq 0} 2^{-k/2} \| (\langle \xi \rangle^s f)* g  \|
_{L_{\tau,\xi}^2(B_k)}  \lesssim \| f \|_{\hat{X}_{(2,1)}^{s,1/2}} 
\|  g \|_{\hat{X}_L^{-1/4,3/4,1}}.
\end{align}
Since $|\xi_2| \sim 2^{k_2-4j}$, 
We use H\"{o}lder's inequality, Young's inequality and the triangle inequality
 to have
\begin{align*}
\text{(L.H.S.)} 
 \lesssim & 2^{3j} \| \langle \xi \rangle^s f \|_{L_{\xi}^2 L_{\tau}^1} 
\| g  \|_{L_{\xi}^1 L_{\tau}^2} \\ 
\lesssim & 2^{3j} \|  f \|_{\hat{X}_{(2,1)}^{s,1/2}} \| |\xi|^{1/4} \|_{L_{\xi_2}^2( |\xi_2| \sim 2^{k_2-4j}  )} \| |\xi|^{-1/4} g \|_{L_{\tau,\xi}^2} \\
\lesssim & 2^{3j} \| f \|_{\hat{X}_{(2,1)}^{s,1/2}} 2^{3k_2/4-3j} \sum_{k_2}
\| |\xi|^{-1/4} g \|_{L_{\tau,\xi}(B_{k_2})} \\
\lesssim & \| f \|_{\hat{X}_{(2,1)}^{s,1/2} } \| g \|_{\hat{X}_L^{-1/4,3/4,1}}.
\end{align*}

(ib) In the case $-7/8 \leq a <-1/4$, we prove 
\begin{align} \label{es_hlh-3}
2^{3j} \sum_{k \geq 0} 2^{-k/2} 
\| (\langle \xi \rangle^s f)* g \|_{ L_{\tau,\xi}^2 (B_k)} 
\lesssim \| f \|_{\hat{X}_{(2,1)}^{s,1/2}} \| g \|_{\hat{X}_L^{a,3a/5+9/10}}.
\end{align}
Since $2^{k_2 } \sim |\xi_2| 2^{4j}$, we use the H\"{o}lder inequality and Young inequality to have
\begin{align*}
(\text{L.H.S.}) \lesssim & 
2^{3j} \| \langle \xi \rangle^s f \|_{L_{\xi}^2 L_{\tau}^1} 
\|  g  \|_{L_{\xi}^1 L_{\tau}^2} \\
\lesssim & 2^{-\frac{12}{5} (a+\frac{1}{4}) j} \| f \|_{\hat{X}_{(2,1)}^{s,1/2}} 
\| |\xi|^{-8a/5-9/10}  \|_{L_{\xi}^2 (|\xi| \lesssim 2^{-3j/2})} 
\|  g \|_{\hat{X}_{L}^{a,3a/5+9/10}},
\end{align*}
which shows the desired estimate since 
$\| |\xi|^{-8a/5-9/10} \|_{L_{\xi}^2 (|\xi| \lesssim 2^{-3j/2})} 
\lesssim 2^{\frac{12}{5} (a+\frac{1}{4}) j}$. 

\vspace{0.5em}

(ii) We consider (\ref{es_hlh-1}) when $g$ is restricted to $D_1$. In this case, $2^{-3j/2} \lesssim 
|\xi_2| \leq 1$ and $ 2^{5j/2} \lesssim  |\tau_2| \lesssim 2^{4j}$.

(iia) Firstly, $g$ is restricted to $B_{[5j/2, 5j/2+ \alpha]}$ with 
$0 \leq \alpha \leq 3j/2$. 
From $2^{-3j/2} \lesssim |\xi| \lesssim 2^{-3j/2+\alpha}$, we use the 
H\"{o}lder inequality and Young inequality to obtain 
\begin{align*}
& \|\xi (\xi^2 f)* g \|_{\hat{X}_{(2,1)}^{s,-1/2}(B_{\geq 2 \alpha })}
\sim  2^{3j} \sum_{k \geq 2\alpha} 2^{-k/2}
\| (\langle \xi \rangle^s f) * g \|_{L_{\tau,\xi}^2(B_k)} \\
& \hspace{0.3cm} \lesssim  2^{3j}~2^{-\alpha} 
\| \langle \xi \rangle^s f  \|_{L_{\xi}^2 L_{\tau}^1}
\| g \|_{L_{\xi}^1 L_{\tau}^2 } \\
& \hspace{0.3cm} \lesssim  2^{ -\frac{12}{5}(a+\frac{1}{4})j} ~2^{-\alpha}
\| f \|_{\hat{X}_{(2,1)}^{s,1/2}} 
\| |\xi|^{-8a/5-9/10} \|_{L_{\xi}^2(2^{-3j/2} \lesssim 
|\xi| \lesssim 2^{-3j/2+\alpha})} \| g \|_{\hat{X}_L^{a,3a/5+9/10}}. 
\end{align*}
In the case $a=-1/4$, the right hand side is bounded by $\sqrt{\alpha}~ 
2^{-\alpha} \| f \|_{\hat{X}_{(2,1)}^{s,1/2}} \| g \|_{\hat{X}_L^{-1/4,3/4}}$ 
because 
\begin{align*} 
\| |\xi_2|^{-1/2} \|_{L_{\xi_2}^2 (2^{-3j/2 } \lesssim |\xi_2| 
\lesssim 2^{2^{-3j/2+\alpha}} )}
\lesssim \sqrt{\alpha}.
\end{align*}
In the case $-7/8 < a <-1/4$, that is bounded by 
$ 2^{-\frac{8}{5} (a+\frac{7}{8}) \alpha} 
\| f \|_{\hat{X}_{(2,1)}^{s,1/2}} \| g \|_{\hat{X}_L^{a,3a/5+9/10}}$ since 
\begin{align*}
\| |\xi|^{-8a/5-9/10} \|_{L_{\xi}^2 (|\xi| \lesssim 2^{-3j/2+\alpha})}
\lesssim 2^{\frac{12}{5} (a+\frac{1}{4}) j}~ 
2^{-\frac{8}{5}(a +\frac{1}{4}) \alpha }.
\end{align*}
We put a sufficiently small number $\varepsilon_3$ such that 
$0< \varepsilon_3 \leq 8(a+7/8)/5$. 
Then we obtain, for $-7/8<a \leq -1/4$, 
\begin{align} \label{es_hlh_a}
\| \xi (\xi^2 f) * g \|_{\hat{X}_{(2,1)}^{s,-1/2} (B_{\geq 2\alpha}) }
\lesssim 2^{-\varepsilon_3 \alpha} \| f \|_{\hat{X}_{(2,1)}^{s,1/2}} 
\| g \|_{\hat{X}_L^{a,3a/5+9/10}}.
\end{align}

(iib) Secondly, $g$ is restricted to $B_{[5j/2+\gamma, 4j]}$ with 
$0 \leq \gamma \leq 2^{3j/2}$. Then we use (\ref{es_dy_3-2}) with $2^{k_2} \sim 2^{j}$ to have
\begin{align*}
\| \xi (\xi^2 f)* g \|_{\hat{X}_{(2,1)}^{s,-1/2}(B_{\leq 2 \alpha} )} \sim & 
2^{3j} \sum_{k \leq 2\alpha} 2^{-k/2} 
\| (\langle \xi \rangle^s f) * g \|_{L_{\tau,\xi}^2(B_k)} \\
\lesssim & 2^{3j/2} \sum_{k \leq 2\alpha} 1~
\| f \|_{\hat{X}_{(2,1)}^{s,1/2}} \| |\xi|^{-1/2} g  \|_{L_{\tau,\xi}^2
(2^{-3j/2+\gamma} \lesssim |\xi|)} \\
\lesssim & \alpha 2^{3j/2} \| f \|_{\hat{X}_{(2,1)}^{s,1/2}} 
\| |\xi|^{-1/2} g \|_{L_{\tau,\xi}^2(2^{-3j/2+\gamma \lesssim |\xi|})}, 
\end{align*}
which is bounded by 
\begin{align*}
\alpha 2^{-\frac{8}{5}(a+\frac{7}{8}) \gamma} \| f \|_{\hat{X}_{(2,1)}^{s,1/2}}
\| g \|_{\hat{X}_L^{a,3a/5+9/10}},
\end{align*}
since $2^{-3j/2+\gamma} \lesssim |\xi_2| \leq 1$ and
\begin{align*}
|\xi_2|^{-a-1/2} \langle \tau_2 \rangle^{-3a/5-9/10}
\sim & |\xi_2|^{-\frac{8}{5}(a+\frac{7}{8})} 
2^{(-\frac{12}{5}a-\frac{18}{5}) j} 
\lesssim  2^{-3j/2}~2^{-\frac{8}{5}(a+ \frac{7}{8}) \gamma }.
\end{align*}
Therefore we obtain 
\begin{align} \label{es_hlh_b}
\| \xi (\xi^2 f)*g \|_{\hat{X}_{(2,1)}^{s,-1/2} (B_{\leq 2 \alpha})} \lesssim 
\alpha 2^{-\varepsilon_3 \gamma} \| f \|_{\hat{X}_{(2,1)}^{s,1/2}} \| g \|_{\hat{X}_L^{a,3a/5+9/10}}.
\end{align}
If $g$ is restricted to 
$B_{[5j/2+\gamma,5j/2+\alpha]}$ with $\gamma < \alpha$, 
from (\ref{es_hlh_a}) and (\ref{es_hlh_b}), we have 
\begin{align} \label{es_hlh_ab}
\| \xi (\xi^2 f)* g \|_{\hat{X}_{(2,1)}^{s,-1/2}} \lesssim \bigl( 2^{-\varepsilon_3 \alpha}+ \alpha 2^{-\varepsilon_3 \gamma} \bigr) 
\|  f \|_{\hat{X}_{(2,1)}^{s,1/2}} \| g \|_{\hat{X}_L^{a,3a/5+9/10}}.
\end{align}
Let $\{ a_n \}_{n=0}^{N}$ be the decreasing sequence defined by 
\begin{align*}
a_0= \frac{3}{2}j, \hspace{0.3cm} a_{n+1}=\frac{1}{2} a_n, \hspace{0.3cm} 
0<a_N \leq \frac{3}{2},
\end{align*}
where $N$ is a minimum integer such that $N \geq \log_2 j$. We first apply with $\alpha=a_0$ and $\gamma=a_1$ and next apply with $\alpha=a_1$ and $\gamma=a_2$. Repeating this procedure at the end we apply with $\alpha=a_N$ and 
$\gamma=0$. From (\ref{es_hlh_ab}), we obtain 
\begin{align*}
\| \xi (\xi^2 f) * g \|_{\hat{X}_{(2,1)}^{s,-1/2}} \lesssim 
\bigl(1+ \sum_{n=0}^N \frac{1}{a_n} \bigr)
\| f \|_{\hat{X}_{(2,1)}^{s,1/2}} \| g \|_{\hat{X}_L^{a,3a/5+9/10}},
\end{align*}
which shows the claim since $\displaystyle \sum_{n=0}^N \frac{1}{a_n}$ is bounded uniformly in $j$. 

\vspace{1em}

 Next, we prove (\ref{BE-Y}) except the case (i). We use the triangle inequality and the Schwarz inequality to have  
\begin{align} \label{es-L-1}
\| f \|_{L_{\tau}^1} \lesssim \sum_{k \geq 0} \| f \|_{L_{\tau}^1 (B_k)} 
\lesssim \sum_{k \geq 0} 2^{k/2}~\| f \|_{L_{\tau}^2 (B_k)}.
\end{align}
From (\ref{es-L-1}), we have, for all $j \neq 0$, 
\begin{align*}
\| \langle \xi \rangle^{s-a} ~|\xi|^{a+1} \langle \tau-\xi^5 \rangle^{-1} (\xi^2 f)* g  \|_{L_{\xi}^2 L_{\tau}^1(A_j)} \lesssim 
\| \xi (\xi^2 f)* g \|_{\hat{X}_{(2,1)}^{s,-1/2}(A_j)}.
\end{align*}
Therefore we obtain (\ref{BE-Y}) for $j \neq 0$ from the proof of 
(\ref{BE-X}). Here 
we only prove (\ref{BE-Y}) in the case (vi). 

\vspace{0.5em}

(VII) Estimate for (vi).  We prove 
\begin{align} \label{Y-vi}
2^{2j_1}~ \| |\xi|^{a+1}~f*g  \|_{L_{\xi}^2 L_{\tau}^1 (A_0)} 
\lesssim \| f \|_{\hat{X}_{(2,1)}^{s,1/2}} \| g \|_{\hat{X}_{(2,1)}^{s,1/2}}.
\end{align}

We consider (\ref{Y-vi}) in the case $|\xi| \leq 2^{-4j_1}$. Since the left hand side of (\ref{Y-vi}) is bounded by $C \| |\xi|^{a+1} \langle \tau \rangle^{-1/4+\varepsilon} f* g \|_{L_{\tau,\xi}^2} $, we obtain the desired estimate in the same manner as (V).
 Thus we only consider the case $2^{-4 j_1} \leq |\xi| \leq 1$ below. 

If $2^{k_{\max}}=2^{k_1}$ or $2^{k_2}$, the left hand side of (\ref{Y-vi}) 
is bounded by  \\
$C \displaystyle 2^{2j_1} \sum_{k \geq 0} 2^{-k/4} 
\| |\xi|^{a+1} f*g  \|_{L_{\tau,\xi}^2(B_k)}$. In the same manner as (V), 
we obtain (\ref{Y-vi}) in this case. 
We consider the case $2^{k_{max}}=2^{k}$. Since $|\xi|^{a+1} 2^{-2sj_1} \lesssim (|\xi| 2^{4j_1} )^{-s/2} \lesssim 2^{k/8} $,  we use the H\"{o}ler inequality and the Young inequality to have
\begin{align*}
\text{(L.H.S)} & \lesssim 2^{-2sj_1+2j_1} \| |\xi|^{a+1} \langle \tau 
\rangle^{-1}
(\langle \xi \rangle^{s } f)* (\langle \xi \rangle^s g)  \|_{L_{\xi}^2 L_{\tau}^1 } \\ 
& \lesssim 2^{2j_1} \| \langle \tau \rangle^{-7/8} 
(\langle \xi \rangle^{s } f) * 
(\langle \xi \rangle^s g)  \|_{L_{\xi}^2 L_{\tau}^1 } \\ 
& \lesssim 
2^{-3j_1/2} \| |\xi|^{-7/8}  \|_{L_{\xi}^2( 2^{-4j_1} \leq |\xi| )}
\| (\langle \xi \rangle^s f )* (\langle \xi \rangle^s g ) 
\|_{L_{\xi}^{\infty} L_{\tau}^1} \\
& \lesssim \| f \|_{\hat{X}_{(2,1)}^{s,1/2}} \| g \|_{\hat{X}_{(2,1)}^{s,1/2}}.
\end{align*}
\end{proof}

\section{Proof of the trilinear estimates}
In this section, we prove the trilinear estimate (\ref{TE-1}). 
This estimate is reduced to some bilinear estimates by using the $[k;Z]$-
multiplier norm method introduced by Tao \cite{Ta}. Here we recall notations and general frame work of the $[k;Z]$-multiplier norm method. For the details, see \cite{Ta}.

Let $Z$ be an abelian additive group with an invariant measure $d \xi$ (for instance $\mathbb{R}^n$, $\mathbb{T}^n$). For any integer $k \geq 2$, we let $\Gamma_{k}(Z)$ denote the hyperplane
\begin{align*}
\Gamma_{k}(Z):=\bigl\{ (\xi_1, \cdots, \xi_k) \in Z^{k}~;~ \xi_1+\cdots+\xi_k
=0 \bigr\}.
\end{align*}
A $[k;Z]-$multiplier is defined to be any function $m~;~\Gamma_{k}(Z) \rightarrow \mathbb{C}$. Then we define the multiplier norm $\| m \|_{[k;Z]}$ to be the best constant such that the inequality
\begin{align*}
\Bigl| \int_{\Gamma_{k}(Z)} m(\xi) \prod_{i=1}^{k} f_i(\xi_i) d\xi_i \Bigr|
\leq C \prod_{i=1}^{k} \| f_i \|_{L^2(Z)},
\end{align*}
for all functions $f_i$ on $Z$. This multiplier norm has the composition rule and the $T T^{*}$ identity as follows.
\begin{lem} \label{com_id}
If $k_1$, $k_2 \geq 1$ and $m_1$, $m_2$ are functions on $Z^{k_1}$ and $Z^{k_2}$ respectively, then 
\begin{align} \label{com_m}
& \| m_{1} (\xi_1, \cdots, \xi_{k_1}) m_2 (\xi_{k_1+1}, \cdots, \xi_{k_1+k_2})
\|_{[k_1+k_2; Z]} \nonumber \\
& \hspace{1.8cm} \leq \|m_1 (\xi_1,\cdots, \xi_{k_1} )\|_{[k_1+1;Z]}
\| m_2(\xi_{1}, \cdots ,\xi_{k_2} )  \|_{[k_2+1,Z]}.
\end{align}
As a special case we have the $T T^{*}$ identity
\begin{align} \label{id_m}
 \| m (\xi_1, \cdots, \xi_{k}) \overline{ m (-\xi_{k+1}, \cdots, -\xi_{2k} ) }
\|_{[2k; Z]} = \|m (\xi_1,\cdots, \xi_{k} )\|_{[k+1;Z]}^2.
\end{align}
for all functions $m~;~Z^{k } \rightarrow \mathbb{R}$.
\end{lem}
For the details, Lemma $3.7$ in \cite{Ta}. 

We estimate (\ref{TE-1}). Schwarz's inequality implies 
\begin{align*}
\| \langle \xi \rangle^{s-a} |\xi|^a f \|_{L_{\xi}^2 L_{\tau}^1} 
\lesssim \| \langle \xi \rangle^{s-a} |\xi|^a \langle \tau-\xi^5 
\rangle^{1/2+\varepsilon} f \|_{L_{\tau,\xi}^2},
\end{align*}
where $\varepsilon >0$ is sufficiently small. Therefore it suffices to show 
\begin{align*}
& \Bigl\|  |\xi_4| \langle \tau_4-\xi_4^5 \rangle^{-1}
\int_{\mathbb{R}^4} f(\tau_1,\xi_1) g(\tau_2,\xi_2)  h(\tau_3,\xi_3) d\tau_1 d\xi_1 d\tau_2 d\xi_2 
\Bigl\|_{\hat{Z}^{s,a}} \\
+& \Bigl\| \langle \xi_4 \rangle^{s-a} |\xi_4|^{a+1} \langle \tau_4-\xi_4^5 
\rangle^{-1/2+\varepsilon} 
\int_{\mathbb{R}^4} f(\tau_1,\xi_1) g(\tau_2,\xi_2)  h(\tau_3,\xi_3) d\tau_1 d\xi_1 d\tau_2 d\xi_2 
\Bigl\|_{L_{\tau_4,\xi_4}^2} \\
\lesssim & 
\| f \|_{\hat{Z}^{s,a}} \|g \|_{\hat{Z}^{s,a}}\| h \|_{\hat{Z}^{s,a}},
\end{align*}
where $\tau_1+\tau_2+\tau_3+\tau_4=0$ and $\xi_1+ \xi_2+\xi_3+\xi_4=0$. 
By symmetry, without loss of generality, we can assume that $|\xi_3| \leq |\xi_2| \leq |\xi_1|$. We put 
\begin{align*}
\Omega_{0}:= \bigl\{ (\vec{\tau}, \vec{\xi}) \in \mathbb{R}^6~;~|\xi_1| \leq 100 \text{ or } |\xi_2|, |\xi_4| \leq 100 \bigr\},
\end{align*}
where $\vec{\tau} =(\tau_1,\tau_2,\tau_3) $ and $\vec{\xi}=(\xi_1,\xi_2,\xi_3)$. Combining the H\"{o}lder inequality and the Young inequality, 
we easily obtain (\ref{TE-1}) in $\Omega_0$. Thus we only consider (\ref{TE-1}) in $\mathbb{R}^6 \setminus \Omega_0$. We divide $\mathbb{R}^6 \setminus \Omega_0$ into five parts as follows. 
\begin{align*} 
\Omega_1:= & \bigl\{(\vec{\tau}, \vec{\xi}) \in \mathbb{R}^6 \setminus \Omega_0~;~|\xi_3| \geq 1 \text{ and } |\xi_4| \geq 1  \bigr\}, \\
\Omega_2:= & \bigl\{( \vec{\tau} ,\vec{\xi}) \in \mathbb{R}^6 \setminus \Omega_0~;~|\xi_3| \geq 1 \text{ and } |\xi_4| \leq 1  \bigr\}, \\
\Omega_3:= & \bigl\{(\vec{\tau}, \vec{\xi}) 
\in \mathbb{R}^6 \setminus \Omega_0~;~
|\xi_2|,~|\xi_4| \geq 1 \text{ and } |\xi_3| \leq 1  \bigr\}, \\
\Omega_4:= & \bigl\{( \vec{\tau}, \vec{\xi}) \in \mathbb{R}^6 \setminus \Omega_0~;~
|\xi_2| \geq 1 \text{ and } |\xi_3|,~ |\xi_4| \leq 1  \bigr\}, \\
\Omega_5:= & \bigl\{(\vec{\tau}, \vec{\xi}) \in \mathbb{R}^6 \setminus \Omega_0~;~
|\xi_1|,~|\xi_4| \geq 1 \text{ and } |\xi_2| \leq 1  \bigr\}.
\end{align*}
We reduce the trilinear inequality by using the composition rule (\ref{com_m}) and the $T T^{*}$ identity (\ref{id_m}). 

\vspace{0.3em}

(A) Estimate in $\Omega_1$. It suffices to show that
\begin{align*}
\Bigl\| \chi_{\Omega_1}~ \frac{\langle \xi_4 \rangle^{s+1}}{\langle \tau_4-\xi_4^5 \rangle^{1/2-\varepsilon} } \prod_{i=1}^3 
\frac{ \langle \xi_i \rangle^{-s}}{ \langle \tau_i-\xi_i^5 \rangle^{1/2}} 
\Bigr\|_{[4; \mathbb{R}^2]} \lesssim 1, 
\end{align*}
where $\varepsilon>0$ is sufficiently small. 
Following $\langle \xi_4 \rangle^{s+1} \lesssim \langle \xi_4 \rangle^{1/2}
\langle \xi_1 \rangle^{s+1/2}$ for $s \geq -1/2$, we use the $T T^{*}$ identity (\ref{id_m}) to have 
\begin{align*}
\text{(L.H.S.)} \lesssim & 
\Bigl\| \chi_{\Omega_1}~ \frac{\langle \xi_4 \rangle^{1/2} 
\langle \xi_1 \rangle^{1/2} }{ \langle \tau_4-\xi_4^5 \rangle^{1/2-\varepsilon} \langle \tau_1-\xi_1^5 \rangle^{1/2} } 
\prod_{i=2}^3 
\frac{ \langle \xi_i \rangle^{-s}}{ \langle \tau_i-\xi_i^5 \rangle^{1/2}} 
\Bigr\|_{[4; \mathbb{R}^2]} \\
\lesssim & 
\Bigl\| \chi_{\{ |\xi_1|,~|\xi_2| \geq 1 \}} (\xi_1,\xi_2) ~ 
\frac{\langle \xi_1 \rangle^{-s} \langle \xi_2 \rangle^{1/2} }{
\langle \tau_1-\xi_1^5 \rangle^{1/2} \langle \tau_2-\xi_2^5 \rangle^{1/2-\varepsilon} } \Bigr\|_{[3; \mathbb{R}^2]}^2. 
\end{align*}
Therefore the trilinear estimate in $\Omega_1$ 
is reduced to the bilinear estimate 
\begin{align} \label{TBE-1}
\| f* g \|_{L_{\tau,\xi}^2 } \lesssim \| p_h f \|_{\hat{X}_{(2,2)}^{s,1/2}} 
\| p_h g \|_{\hat{X}_{(2,2)}^{-1/2,1/2-\varepsilon}},
\end{align}
where $\hat{X}_{(2,2)}^{s,b}$ is defined by the norm 
\begin{align*}
\| f \|_{\hat{X}_{(2,2)}^{s,b}}:= 
\| \langle \xi \rangle^s \langle \tau-\xi^5 \rangle^b 
f \|_{L_{\tau,\xi}^2} \text{ for } s,b \in \mathbb{R}.
\end{align*}

(B) Estimate in $\Omega_2$. It suffices to show that 
\begin{align*}
\Bigl\| \chi_{\Omega_2}~ \frac{|\xi_4|^{a+1}}{\langle \tau_4-\xi_4^5 \rangle^{1/4} } \prod_{i=1}^3 
\frac{ \langle \xi_i \rangle^{-s}}{ \langle \tau_i-\xi_i^5 \rangle^{1/2}} 
\Bigr\|_{[4; \mathbb{R}^2]} \lesssim 1. 
\end{align*}
We use the composition rule (\ref{com_m}) to have 
\begin{align*}
\text{(L.H.S.)} \lesssim & 
\Bigl\| \chi_{ \{|\xi_1| \leq 1, |\xi_2| \geq 1 \}}(\xi_1,\xi_2) ~ 
\frac{|\xi_1|^{a+1} \langle \xi_2 \rangle^{-s}}
{\langle \tau_1-\xi_1^5 \rangle^{1/4} \langle \tau_2-\xi_2^5 \rangle^{1/2} } 
\Bigr\|_{[3; \mathbb{R}^2]} \\
& \hspace{0.6cm} \times \Bigl\| \chi_{\{ |\xi_1|, |\xi_2| \geq 1 \}} 
(\xi_1,\xi_2) \prod_{i=1}^2
\frac{ \langle \xi_i \rangle^{-s}}{ \langle \tau_i-\xi_i^5 \rangle^{1/2}} 
\Bigr\|_{[3; \mathbb{R}^2]},
\end{align*}
which shows that the trilinear estimate in $\Omega_2$ is reduced to 
\begin{align} \label{TBE-2}
\| f* g \|_{L_{\tau,\xi}^2} \lesssim \| p_l f \|_{\hat{X}_L^{-a-1,1/4}}
\| p_h g \|_{\hat{X}_{(2,2)}^{s,1/2}}
\end{align}
and 
\begin{align*}
\| f*g \|_{L_{\tau,\xi}^2} \lesssim \| p_h f \|_{\hat{X}_{(2,2)}^{s,1/2}}
\| p_h g \|_{\hat{X}_{(2,2)}^{s,1/2}}. 
\end{align*}

(C) Estimate in $\Omega_3$. It suffices to show that 
\begin{align*}
\Bigl\| \chi_{\Omega_3}~ \frac{\langle \xi_4 \rangle^{s+1}}{\langle \tau_4-\xi_4^5 \rangle^{1/2-\varepsilon} } \prod_{i=1}^2 
\frac{ \langle \xi_i \rangle^{-s}}{ \langle \tau_i-\xi_i^5 \rangle^{1/2}} 
 \frac{ |\xi_3|^{-a} }{ \langle \tau_3-\xi_3^5 \rangle^{3/8}} 
\Bigr\|_{[4; \mathbb{R}^2]} \lesssim 1. 
\end{align*}
Following $\langle \xi_4 \rangle^{s+1} \lesssim \langle \xi_4 \rangle^{1/2}
\langle \xi_1 \rangle^{s+1/2} $ for $s \geq -1/2$, 
we use the composition rule (\ref{com_m}) to obtain 
\begin{align*}
\text{(L.H.S.)} \lesssim & 
\Bigl\| \chi_{\Omega_3}~ \frac{\langle \xi_4 \rangle^{1/2} 
\langle \xi_1 \rangle^{1/2} \langle \xi_2 \rangle^{-s} |\xi_3|^{-a} }
{\langle \tau_4-\xi_4^5 \rangle^{1/2-\varepsilon} \langle \tau_1-\xi_1^5 
\rangle^{1/2} \langle \tau_2-\xi_2^5 \rangle^{1/2} 
\langle \tau_3-\xi_3^5 \rangle^{3/8} } \Bigr\|_{[4; \mathbb{R}^2]} \\
\lesssim & \Bigl\| \chi_{ \{|\xi_1| \leq 1, |\xi_2| \geq 1 \}}(\xi_1,\xi_2) ~ 
\frac{|\xi_1|^{-a} \langle \xi_2 \rangle^{1/2}}
{\langle \tau_1-\xi_1^5 \rangle^{3/8} \langle \tau_2-\xi_2^5 \rangle^{1/2-\varepsilon} } \Bigr\|_{[3; \mathbb{R}^2]} \\
& \hspace{0.6cm} \times \Bigl\| \chi_{\{ |\xi_1|, |\xi_2| \geq 1 \}} 
(\xi_1,\xi_2) \frac{\langle \xi_1 \rangle^{1/2} \langle \xi_2 \rangle^{-s}}
{ \langle \tau_1-\xi_1^5 \rangle^{1/2-\varepsilon} 
\langle \tau_2-\xi_2^5 \rangle^{1/2}}  \Bigr\|_{[3; \mathbb{R}^2]},
\end{align*}
which implies that the trilinear estimate in $\Omega_3$ is reduced to 
(\ref{TBE-1}) and 
\begin{align} \label{TBE-3}
\| f*g \|_{L_{\tau,\xi}^2} \lesssim  \| p_l f \|_{\hat{X}_L^{a,3/8}}
\| p_h g \|_{\hat{X}_{(2,2)}^{-1/2,1/2-\varepsilon}}.
\end{align}

Similar to above, in other cases, the trilinear estimate is reduced to the bilinear estimates (\ref{TBE-1}), (\ref{TBE-2}) and (\ref{TBE-3}). 
We remark that Chen, Li, Miao and Wu \cite{CLMW} 
proved (\ref{TBE-1}) for $s \geq -1/4$ by 
using the block estimates established by Tao \cite{Ta}. For the details, see 
Lemma $5.2$ in \cite{CLMW}. 
Thus we omit the proof of (\ref{TBE-1}) and give the proofs of (\ref{TBE-2}) and (\ref{TBE-3}). 
From $L_{\xi}^2$-property of $\hat{X}_{(2,2)}^{s,b}$ and $\hat{X}_{L}^{a,b}$, it suffices to show two lemmas as follows. 

\begin{lem} \label{lem_TE_2}
Let $s \geq -1/4$ and $-3/2 <a \leq -1/4$. Suppose that $f$ is supported on $A_0$ and $g$ is supported on $A_{j_2}$ for $j_2 >0$. Then we have, for $j \geq 0$,
\begin{align} \label{TE-L-1}
\| f*g \|_{L_{\tau,\xi}^2 (A_j) } \lesssim C(j,j_1,j_2)
\| p_l f \|_{\hat{X}_{(2,2)}^{-a-1,1/4}}
\| p_h g \|_{\hat{X}_{(2,2)}^{s,1/2}},
\end{align}
in the cases (i) and (vii) of Proposition~\ref{prop_BE-2}.
\end{lem}

\begin{lem} \label{lem_TE_3}
Let $s \geq -1/4$ and $-3/2<a \leq -1/4$. Suppose that $f$ is supported on 
$A_{0}$ and $g$ is supported on $A_{j_2}$ for $j_2 >0$. 
Then we have, for $j \geq 0$, 
\begin{align} \label{TE-L-2}
\| f*g \|_{L_{\tau,\xi}^2(A_j) } \lesssim C(j,j_1,j_2)
\| p_l f \|_{\hat{X}_{(2,2)}^{a,3/8} }
\| p_h g \|_{\hat{X}_{(2,2)}^{-1/2, 1/2-\varepsilon}},
\end{align}
in the cases (i) and (vii) of Proposition~\ref{prop_BE-2}.
\end{lem}

Here we define $2^{k_{\max}} \geq 2^{k_{\text{med}}} \geq 2^{k_{\min}}$ to be 
the maximum, median and minimum of $2^{k}, 2^{k_1}, 2^{k_2}$ respectively. 

\begin{proof}[Proof of Lemma~\ref{lem_TE_2}]
(I) Estimate for (i). We use the H\"{o}lder inequality and the Young inequality  to have 
\begin{align*} 
\| f* g \|_{L_{\tau,\xi}^2} \lesssim \| f \|_{L_{\xi}^1 L_{\tau}^{3/2}}
\| g \|_{L_{\xi}^2 L_{\tau}^{6/5}}
\lesssim & \| |\xi_1|^{a+1} \|_{L_{\xi_1}^{2}(|\xi_1| \leq 1) } \| |\xi|^{-a-1} f \|_{L_{\xi}^2 L_{\tau}^{3/2}} \| g \|_{L_{\xi}^2 L_{\tau}^{6/5}} \\
\lesssim & \| f \|_{\hat{X}_L^{a,1/6+\varepsilon}} 
\| g \|_{\hat{X}_{(2,2)}^{0,1/3+\varepsilon}},
\end{align*}
which shows the required estimate.

\vspace{0.3em}

(II) Estimate for (vii).  We prove 
\begin{align} \label{TE-L-1_2}
\| p_h (f* g) \|_{L_{\tau,\xi}^2(A_j)} \lesssim \| p_l f \|_{\hat{X}^{-a-1,1/4}}\| p_h g \|_{\hat{X}_{(2,2)}^{s,1/2-\varepsilon}}.
\end{align}

(IIa) We consider (\ref{TE-L-1_2}) when $f$ is restricted to 
$\{ (\tau,\xi)~;~|\xi| \leq 2^{-2j_1} \}$. 
We use the H\"{o}lder inequality and the Young inequality to have 
\begin{align*}
\| f* g\|_{L_{\tau,\xi}^2(A_j)} \sim & 2^{-s j} \| f * (\langle \xi \rangle^{s} g) \|_{L_{\tau,\xi}^2} \\
\lesssim & 2^{-s j} \| f \|_{L_{\xi}^1 L_{\tau}^{3/2}} \| \langle \xi \rangle^s  g \|_{L_{\xi}^2 L_{\tau}^{6/5}} \\
\lesssim & 2^{-s j} \| |\xi_1|^{a+1} \|_{L_{\xi_1}^2 (|\xi_1| \leq 2^{-2j_1})} 
\| |\xi|^{-a-1} f \|_{L_{\xi}^2 L_{\tau}^{3/2}} 
\| g \|_{\hat{X}_{(2,2)}^{s,1/2}} \\
\lesssim & 2^{-(s+2a+3)j } \| f \|_{\hat{X}_L^{-a-1,1/4}} 
\| g \|_{\hat{X}_{(2,2)}^{s,1/2}}, 
\end{align*}
which implies the desired estimate. 

\vspace{0.3em}

(IIb) We prove (\ref{TE-L-1_2}) when $f$ is restricted to $\{(\tau,\xi)~;~ 
2^{-2j} \leq |\xi| \leq 1 \}$. 

\vspace{0.3em}

(IIb-1) We consider the case $2^{k_{\max}} \sim |\xi_1| 2^{4j}$. From $a+1 \geq -s/2$ and $s \geq -1/4$, we have 
\begin{align*}
|\xi_1|^{a+1} 2^{-sj} \lesssim (|\xi_1| 2^{2j})^{-s/2} 
\lesssim 2^{j/4} |\xi_1|^{1/8} \lesssim 
2^{2j} 2^{-7k_{\max}/16} |\xi_1|^{9/16}.
\end{align*}
Then we obtain 
\begin{align*}
\| f*g \|_{L_{\tau,\xi}^2 (A_j)} \lesssim & 2^{-sj} \| |\xi_1|^{a+1} 
(|\xi|^{-a-1} f)* (\langle \xi \rangle^s g) \|_{L_{\tau,\xi}^2} \\
\lesssim & 2^{2j} 2^{-7k_{\max}/16} \| (|\xi|^{-a-1} f)* (\langle \xi \rangle^s g) \|_{L_{\tau,\xi}^2} \\
\lesssim & 2^{2j} \|(|\xi|^{-a-1} \langle \tau \rangle^{-5/16} f)* 
(\langle \xi \rangle^s \langle \tau-\xi^5 \rangle^{-1/8} g) 
\|_{L_{\tau,\xi}^2},
\end{align*}
which shows the required estimate by using (\ref{es_dy_1-2}) with $K \sim 2^j$.

\vspace{0.3em}

(IIb-2) We consider in other cases, namely $2^{k_{\max}} \sim 2^{k_{\text{med}}} \gg |\xi_1| 2^{4j}$. We only prove (\ref{TE-L-1_2}) in the most difficult case $2^{k_{\max}} =2^{k}$ and $2^{k_{\text{med}}} =2^{k_1}$. 
Following 
\begin{align*}
|\xi_1|^{a+1} 2^{-sj} \sim (|\xi_1| 2^{2j})^{-s/2-1/8} |\xi_1|^{1/8} 2^{j/4} 
\lesssim (|\xi_1| 2^{4j})^{1/16} |\xi_1|^{1/8} \lesssim 2^{k_1/16},
\end{align*}
we use the H\"{o}lder inequality and the Young inequality to have
\begin{align*}
\| f*g \|_{L_{\tau,\xi}^2 (A_j)} \sim & 2^{-sj} \| |\xi_1|^{a+1} 
(|\xi|^{-a-1} f)* (\langle \xi \rangle^s g ) \|_{L_{\tau,\xi}^2} \\
\lesssim & \| (|\xi|^{-a-1} \langle \tau \rangle^{1/16} f) *
(\langle \xi \rangle^{s} g ) \|_{L_{\tau,\xi}^2} \\
\lesssim & \| |\xi|^{-a-1} \langle \tau \rangle^{1/16} f \|_{L_{\xi}^1 
L_{\tau}^{3/2}} \| \langle \xi \rangle^{s} g \|_{L_{\xi}^2 L_{\tau}^{6/5}} \\
\lesssim & \| f \|_{\hat{X}_L^{-a-1,1/4}} \| g \|_{\hat{X}_{(2,2)}^{s,1/2}}.
\end{align*}
\end{proof}

\begin{proof}[Proof of Lemma~\ref{lem_TE_3}] 
(I) Estimate for (i). We use the H\"{o}lder inequality and the Young inequality to have 
\begin{align*}
\|f*g \|_{L_{\tau,\xi}^2 } \lesssim \| f \|_{L_{\xi}^{1} L_{\tau}^{3/2}}
\| g \|_{L_{\xi}^2 L_{\tau}^{6/5}} 
\lesssim \| f \|_{\hat{X}_L^{0,1/6+\varepsilon}} 
\| g \|_{\hat{X}_{(2,2)}^{0,1/3+\varepsilon}},
\end{align*}
which shows the desired estimate.

\vspace{0.3em}

(II) Estimate for (vii). We prove 
\begin{align} \label{TE-L-2_1}
\| p_h(f *g) \|_{L_{\tau,\xi}^2(A_j)} \lesssim 
\| p_l f \|_{\hat{X}_L^{a,3/8}} 
\| p_h g \|_{\hat{X}_{(2,2)}^{-1/2,1/2-\varepsilon}}.
\end{align}

(IIa) We consider (\ref{TE-L-2_1}) when $2^{k_{\max}} \sim |\xi_1| 2^{4j}$. 
Following 
\begin{align*}
|\xi_1|^{-a} 2^{j/2} \sim 2^{2j} (|\xi_1| 2^{4j})^{-3/8} |\xi_1|^{-a+3/8}
\lesssim 2^{2j} 2^{-3k_{\max}/8},
\end{align*}
we use (\ref{es_dy_1-2}) with $K \sim 2^{j}$ to have 
\begin{align*}
\| f*g \|_{L_{\tau,\xi}^2 (A_j)} \sim & 2^{j/2} \| f *(\langle \xi \rangle^{-1/2} g) \|_{L_{\tau,\xi}^2} \\
\lesssim & 2^{2j} 2^{-3 k_{\max}/8} \| (|\xi|^a f) * (\langle \xi \rangle^{-1/2}g ) \|_{L_{\tau,\xi}^2} \\
\lesssim & 2^{2j} \| (|\xi|^{a} \langle \tau \rangle^{-1/4} f) * 
(\langle \xi \rangle^{-1/2} \langle \tau-\xi^5 \rangle^{-1/8} g) \|_{L_{\tau,\xi}^2} \\
\lesssim & \| f \|_{\hat{X}_L^{a,3/8}} \| g \|_{\hat{X}_{(2,2)}^{-1/2,1/2+\varepsilon}}.
\end{align*}

\vspace{0.3em}

(IIb) We prove (\ref{TE-L-2_1}) 
in the case $2^{k_{\max}} \sim 2^{k_{\text{med}}} 
\gg |\xi_1|2^{4j}$. It suffices to show (\ref{TE-L-2}) in the case 
$2^{k_{\max}} =2^{k} $ and $2^{k_{\text{med}}} =2^{k_1}$. 
Following 
\begin{align*}
2^{j/2} \sim (|\xi_1| 2^{4j})^{1/8} |\xi_1|^{-1/8} \lesssim |\xi_1|^{-1/8}
2^{k_1/8},
\end{align*}
we use the H\"{o}lder inequality and the Young inequality to have
\begin{align*}
\| f*g \|_{L_{\tau,\xi}^2 (A_j)} \lesssim & \| (|\xi|^{-1/8} \langle \tau \rangle^{1/8} f)*(\langle \xi \rangle^s g) \|_{L_{\tau,\xi}^2} \\
\lesssim & \| |\xi|^{-1/8} \langle \tau \rangle^{1/8} f 
\|_{L_{\xi}^1 L_{\tau}^{3/2}} \| \langle \xi \rangle^{-1/2} g \|_{L_{\xi}^2 L_{\tau}^{6/5}} \\
\lesssim & \| |\xi_1|^{-a-1/8} \|_{L_{\xi_1}^2 (|\xi_1| \leq 1)} 
\| |\xi|^a \langle \tau \rangle^{1/8} f \|_{L_{\xi}^2 L_{\tau}^{3/2}}
\| g \|_{\hat{X}_{(2,2)}^{-1/2, 1/2 -\varepsilon}} \\
\lesssim & \| f \|_{\hat{X}_L^{a,3/8}} 
\| g \|_{\hat{X}_{(2,2)}^{-1/2,1/2-\varepsilon}}.
\end{align*}
\end{proof}

\section{Proof of the main results}
In this section, we give the proof of the main theorems. The function space $Z_{T}^{s,a}$ is defined by the norm 
\begin{align*}
\|u \|_{Z_T^{s,a}}:=\inf \bigl\{ \| v \|_{Z^{s,a}}~;~u(t)=v(t)~\text{on}~
t \in [0,T] \bigr\}.
\end{align*}
We obtain the following well-posedness result.
\begin{prop} \label{prop_well}
Let $s,a$ satisfy (\ref{co_op}) and $r>1$. 

(Existence) For any $u_0 \in B_r(H^{s,a})$, there exist $T \sim r^{-10/(3+2a)}$ and \\
$u \in C([0,T];H^{s,a}) \cap Z^{s,a}_T$ satisfying the following integral form for (\ref{5KdV});
\begin{align} \label{integral-1}
u(t)= & U(t) u_0 -c_1 \int_{0}^t U(t-s) \p_x (u(s))^3 ds  
\nonumber \\ 
- & c_2 \int_{0}^t U(t-s) \p_x (\p_x u(s))^2 ds 
-c_3 \int_{0}^t U(t-s) \p_x (u \p_x^2 u(s)) ds
\end{align}
Moreover the data-to-solution map, 
$B_r(H^{s,a}) \ni u_0 \mapsto u \in C([0,T]; H^{s,a}) \cap Z^{s,a}_{T}$, is Lipschitz continuous. 

(Uniqueness) Assume that $u,v \in  C([0,T]; H^{s,a}) \cap Z^{s,a}_{T}$ satisfy (\ref{integral-1}). Then $u(t)=v(t)$ on $t \in [0,T]$. 
\end{prop}

\begin{proof}
We first prove the existence of the solution of (\ref{integral-1}). 
This equation is the scale invariant with respect to 
the following scaling.
\begin{align*}
u(t,x) \mapsto u_{\lambda}(t,x):=\lambda^{-2} u(\lambda^{-5}t,\lambda^{-1}x), 
\hspace{0.3cm} \lambda \geq 1.
\end{align*}
A direct calculation shows 
\begin{align} \label{scale}
\| u_{\lambda}(0, \cdot) \|_{H^{s,a}} 
\leq \lambda^{-3/2-a} \| u_0 \|_{H^{s,a}}.
\end{align}
Therefore we can assume that initial data is small enough. 
Here we use propositions~\ref{prop_mult-ES},~\ref{prop_linear1} and \ref{prop_linear2} to construct the solution by the fixed point argument. For details, see 
the proof of Proposition $4.1$ in \cite{KT}.

We next prove the uniqueness of solutions by the argument in \cite{MT}. We define the space $W^{s,a}$ with the norm 
\begin{align*}
\| u \|_{W^{s,a}}:=\| u \|_{Z^{s,a}}+ \| u \|_
{L_t^{\infty} (\mathbb{R}; H^{s,a})}.
\end{align*}
In the same manner as the proof of Theorem 2.5 in \cite{MT}, we obtain, for 
$1/2 < b <1$, 
\begin{align} \label{MT_L}
w \in X_{(1,1),T_{\lambda}}^{s,a,b}, \hspace{0.3cm} w(0,x)=0 
\Rightarrow \lim_{\delta \rightarrow +0} 
\| w |_{[0,\delta]} \|_{X_{(1,1),\delta}^{s,a,b}}=0,
\end{align}
where $T_{\lambda}:=\lambda^{5} T$, $\lambda \geq 1$ and the function space
 $X_{(1,1)}^{s,a,b}$ defined by the norm
\begin{align*}
\| u \|_{X_{(1,1)}^{s,a,b}}:=
\bigl\| \bigl\{  \| \langle \xi \rangle^{s-a} |\xi|^a 
\langle \tau-\xi^5 \rangle^b \widehat{u}  \|_{L_{\tau,\xi}^2 (A_j \cap B_k) }  
\bigr\}_{j,k \geq 0} \bigr\|_{l_{j,k}^1}.
\end{align*}
Let $u \in W^{s,a}$ and $u(0,x)=0$. 
Since $W^{s,a}$ contains $\mathcal{Z}(\mathbb{R}^2)$ densely, 
We can choose $v \in \mathcal{Z} $ satisfying 
$\| u-v  \|_{W^{s,a}} < \varepsilon$ where $\varepsilon$ is an arbitrary positive number. Now we have 
\begin{align*}
\|v(0) \|_{H^{s,a}} =\| (u-v) (0)  \|_{H^{s,a}} \lesssim \| u-v \|_{W^{s,a}}
< \varepsilon.
\end{align*}
Note that
\begin{align} 
\sup_{t \in \mathbb{R}} \| u(t) \|_{H^{s,a}} \lesssim \| u \|_{W^{s,a}}
 \lesssim \| u \|_{X^{s,a,b}}
\end{align}
for any $3/4 < b < 1$. From the above argument, we obtain 
\begin{align*}
\| u \|_{W_T^{s,a}} \lesssim & \| u-v \|_{W^{s,a}}+
 \| v-U(t)v(0) \|_{W_T^{s,a}}+\| U(t) v(0) \|_{X^{s,a,b}} \\
\lesssim & \varepsilon +\| v-U(t) v(0) \|_{X_{(1,1),T}^{s,a,b}}+\| v(0) \|_{H^{s,a}} \\
\lesssim & \varepsilon +\|v-U(t) v(0) \|_{X_{(1,1), T}^{s,a,b}}.
\end{align*}
The second term tends to $0$ as $T \rightarrow 0$ from (\ref{MT_L}), which shows that 
\begin{align} \label{uni_W}
\lim_{T \rightarrow 0} \| u \|_{W_T^{s,a}}=0.
\end{align}
Combining Propositions~\ref{prop_mult-ES},~\ref{prop_linear1} and \ref{prop_linear2} and (\ref{uni_W}), we have the uniqueness. For the details, see \cite{Ki09}.
\end{proof}

We next prove a priori estimate (\ref{apr-es}). The proof is based on Tsugawa's work \cite{Ts}. 
\begin{proof}[Proof of Proposition~\ref{prop_apr}]
By the density argument, without loss of generality, we can assume $u \in \mathcal{Z}$. 
We put the Fourier multiplier $P$ defined by 
\begin{align*}
Pu:=\mathcal{F}_{\xi}^{-1} |\xi|^a~ \chi_{\{|\xi| \leq 1 \}} (\xi)  \mathcal{F}_{x} u.
\end{align*}
Calculating 
\begin{align*}
\int P \bigl(\p_t u-\p_x^5 u -\frac{2}{5} \alpha^2 \p_x (u)^3 + \alpha \p_x (\p_x u)^2 + 2\alpha \p_x ( u \p_x^2 u) \bigr) \cdot Pu dx=0,
\end{align*}
we have
\begin{align*}
\int \p_t Pu \cdot Pu dx - \int P\p_x^5 u \cdot Pu dx - \frac{2}{5} \alpha^2 
\int P \p_x(u)^3 \cdot Pu dx \\
 -\alpha \int P \p_x(\p_x u)^2 \cdot Pu dx+
\alpha \int P \p_x^3 (u)^2 \cdot Pu dx =0.
\end{align*}
The second term of the right hand side vanishes. 
We note
\begin{align} \label{P-pro}
\widehat{P \p_x} \leq |\xi|^{a+1} |_{|\xi| \leq 1} \leq 1.
\end{align}
for $a \geq -1$. By the Sobolev inequality and (\ref{P-pro}), the third term is bounded by 
\begin{align*}
\| u^2 \|_{L^1} \| u \|_{L^{\infty}} \| P^2 \p_x u  \|_{L^{\infty}}
\lesssim & \| u \|_{L^2}^{5/2} \| \p_x u \|_{L^2}^{1/2} 
\|P^2 \p_x u  \|_{L^2}^{1/2} \| (P \p_x)^2 u \|_{L^2}^{1/2} \\
\lesssim & \| u \|_{L^2}^{3} \| \p_x u \|_{L^2}^{1/2} \|Pu  \|_{L^2}^{1/2}.
\end{align*}
Similarly, the fourth term is bounded by 
\begin{align*}
\| (\p_x u)^2 \|_{L^1} \| P^2 \p_x u \|_{L^{\infty}}
\lesssim \| \p_x u \|_{L^2}^2 \| u \|_{L^2}^{1/2} \| Pu \|_{L^2}^{1/2},
\end{align*}
and the fifth term is bounded by 
\begin{align*}
\| u^2 \|_{L^1} \| P^2 \p_x^3 u \|_{L^{\infty}} \lesssim 
\| u^2 \|_{L^2}^{5/2} \| Pu \|_{L^2}^{1/2}.
\end{align*}
Following the above estimates, we obtain 
\begin{align*}
\p_t \| Pu \|_{L^2}^{3/2} \lesssim \| u \|_{L^2}^3 \| \p_x u \|_{L^2}^{1/2}
+\| u \|_{L^2}^{1/2} \|\p_x u \|_{L^2}^2 + \| u \|_{L^2}^{5/2}.
\end{align*}
Therefore we have
\begin{align} \label{Pu_es}
\sup_{0 \leq t \leq T} 
\| Pu (t,\cdot) \|_{L^2}^{3/2} \leq \|P u_0 \|_{L^2}^{3/2}+CT
\bigl(\| u \|_{L^2}^{3} \| \p_x u \|_{L^2}^{1/2} +\| u \|_{L^2}^{1/2} 
\| \p_x u \|_{L^2}^2 + \| u \|_{L^2}^{5/2} \bigr).
\end{align}
(\ref{5KdV}) is complete integrable in the case $c_1=-2\alpha^2/5$, $c_2=\alpha$ and $c_3=2\alpha$ with $\alpha \in \mathbb{R} \setminus \{0 \}$. So this equation particularly has the conserved quantities as follows:
\begin{align} \label{L^2_law}
\| u (t,\cdot) \|_{L^2} & = \| u_0 \|_{L^2}, \\
\label{H^1_law}
\int (\p_x u)^2+ \frac{2}{5} \alpha u^3 dx &= \int (\p_x u_0)^2+
\frac{2}{5} \alpha u_0^3 dx.
\end{align}
Using the Sobolev inequality and (\ref{L^2_law}) to (\ref{H^1_law}), we have 
\begin{align} \label{H^1_es}
\| \p_x u (t, \cdot) \|_{L^2} \lesssim \| \p_x u_0 \|_{L^2}^2+ 
\| u_0 \|_{L^2}^{10/3}.
\end{align}
Substituting (\ref{L^2_law}) and (\ref{H^1_es}) into (\ref{Pu_es}), 
we have
\begin{align} \label{low_es}
\sup_{0 \leq t \leq T} \|Pu (t,\cdot)  \|_{L^2}^{3/2} \leq 
\|Pu_0 \|_{L^2}^{3/2}+ CT \bigr( \| u_0 \|_{L^2}^{15/4}+ \| u_0 \|_{H^1}^{5/2} \bigr). 
\end{align}
Since 
\begin{align*}
\| u (t,\cdot) \|_{H^{1.a}}^2 \leq \| Pu (t,\cdot) \|_{L^2}^2+ \|u (t,\cdot)\|_{L^2}^2+  \| \p_x u (t,\cdot) \|_{L^2}^2,
\end{align*}
we obtain (\ref{apr-es}) from (\ref{L^2_law}), (\ref{H^1_es}) and (\ref{low_es}). 
\end{proof}

Finally, we prove Theorem~\ref{thm_ill}. 
We first prove (i) in Theorem~\ref{thm_ill}. 
In \cite{BT}, Bejenaru and Tao, for the quadratic Schr\"{o}dinger equation with nonlinear term $u^2$, proved the discontinuity of the data-to-solution map for any $s<-1$. We essentially follow their argument to obtain the following proposition. 
\begin{prop} \label{prop_ill} 
Let $s <s_a:=-2a-2$, $-3/2<a < -7/8$, $c_2 \neq c_3$ and $0 <\delta \ll 1$. Then there exist $T=T(\delta)>0$ and 
a sequence of initial data 
$\{ \phi_{N,\delta}  \}_{N=1}^{\infty} \in H^{\infty}$ satisfying the following three conditions for any $t \in (0,T]$, 
 
 (1) $ \| \phi_{N, \delta} \|_{H^{s_a,a}} \sim \delta $, 

 (2) $\| \phi_{N, \delta} \|_{H^{s,a}} \rightarrow 0$ as $N \rightarrow \infty$, 

 (3) $\| u_{N,\delta} (t) \|_{H^{s,a}} \gtrsim \delta^2$, \\
where $u_{N, \delta}(t)$ is the solution to (\ref{5KdV}) 
 obtained in Proposition~\ref{prop_well} with the initial data 
$\phi_{N,\delta}$.
\end{prop}
\begin{proof} 
Let $N \gg 1$. 
We put the initial data $ \phi_{N,\delta}$ as follows:
\begin{align*}
\phi_{N, \delta}(x)= \delta N^{2a+4} \cos(N x) \int_{-\gamma}^{\gamma} 
e^{i \xi x} d\xi.
\end{align*}
where $\gamma:=N^{-4}$. 
A simple calculation shows that
\begin{align} \label{initial-Fo}
\widehat{\phi}_{N, \delta}(\xi) \sim \delta N^{2a+4} \chi_{B^{+}}(\xi)
+\delta N^{2a+4} \chi_{B^{-}}(\xi),
\end{align}
where
\begin{align*}
B^{\pm}:=[\pm N-\gamma,~ \pm N+\gamma].
\end{align*}
Thus we have 
\begin{align} \label{initial-norm}
\| \phi_{N,\delta}  \|_{H^{s,a}} \sim \delta N^{s+2a+2}, \hspace{0.3cm}
\| U(t) \phi_{N, \delta} \|_{H^{s,a}}=\| \phi_{N,\delta} \|_{H^s} \sim \delta N^{s+2a+2 } . 
\end{align}
Since $\| \phi_{N,\delta} \|_{H^{s_a,a}} \sim \delta$, we have 
$T=T(\delta)>0$ and the solution $u_{N, \delta}$ to (\ref{5KdV}) with the initial data 
$\phi_{N,\delta}$ by Proposition~\ref{prop_well}. 
Let $t \in (0, T]$. 
A direct calculation shows that 
\begin{align} \label{qua-Fo}
 \widehat{A_2}(u_0)( t)= & (c_2-c_3) \exp (i \xi^5 t) \int 
\frac{1- \exp(-iq_1t) }{q_1} 
~\xi \xi_1 (\xi-\xi_1)
~\widehat{u}_0(\xi_1) \widehat{u}_0 (\xi-\xi_1) d\xi_1 \nonumber \\
+ & \frac{c_3}{2} \exp (i \xi^5 t) \int 
\frac{1- \exp(-iq_1t) }{q_1} 
~\xi^3
\widehat{u}_0(\xi_1) \widehat{u}_0 (\xi-\xi_1) d\xi_1 \nonumber \\
:= & \widehat{A}_{2,1}(u_0) (t) + \widehat{A}_{2,2}(u_0) (t),
\end{align}
where
\begin{align*}
 q_1:=\frac{5}{2} \xi \xi_1(\xi-\xi_1) \bigl\{ \xi^2 +\xi_1^2 +(\xi-\xi_1)^2
\bigr\}.  
\end{align*}
By similarly argument to the proof of Theorem $1.2$ in \cite{KT}, 
substituting (\ref{initial-Fo}) into 
(\ref{qua-Fo}), we obtain for $c_2 \neq c_3$
\begin{align} \label{qua-norm}
\| A_2(\phi_{N, \delta}) (t)  \|_{H^{s,a}} \gtrsim \delta^2.
\end{align}
Now we put $v_{N,\delta}(t):= u_{ N,\delta}(t) -U(t)\phi_{N,\delta}-
A_2(\phi_{N ,\delta})(t)$. 
Since the data-to-solution map is Lipschitz continuous with $s=s_a$, 
we obtain 
\begin{align} \label{er-norm}
\|v_{N,\delta} (t)  \|_{H^{s_a,a}} \lesssim \delta^3
\end{align}
by Propositions~\ref{prop_mult-ES}, \ref{prop_linear1} and \ref{prop_linear2}. 
From (\ref{initial-norm}), (\ref{qua-norm}) and (\ref{er-norm}), we obtain 
\begin{align*} 
\| u_{ N,\delta}(t) \|_{H^{s,a}} \geq 
\| A_2 (\phi_{N ,\delta})(t) \|_{H^{s,a}}- 
\|v_{N,\delta}(t) \|_{H^{s,a} } - 
\| U(t) \phi_{N, \delta} \|_{H^{s,a}} \gtrsim \delta^2,
\end{align*}
for all $N \gg 1$. Since $\| \phi_{N,\delta} \|_{H^{s,a}} \rightarrow 0$ as $N \rightarrow \infty$, this shows the discontinuity of the flow map.
\end{proof}

Secondly, we prove Theorem~\ref{thm_ill} (ii). By the general argument in \cite{Ho}, it suffices to show the following estimate fails for $|t|$ bounded.
\begin{align*} 
\| A_2(u_0) (t)  \|_{H^{s,a}}^2 \lesssim \| u_0 \|_{H^{s,a}}^2.
\end{align*}
We put the initial data $\{ \psi_N \}_{N=1}^{\infty} \in H^{\infty}$ 
as follows:
\begin{align*} 
\psi_N(x):=N^{-s+2} \cos(Nx) \int_{-\gamma}^{\gamma} e^{i \xi x} d\xi
+N^{4a+2} \cos(N^{-4} x) \int_{-\gamma/2}^{\gamma/2} e^{i \xi x} d\xi.
\end{align*}
A direct computation shows that 
\begin{align} \label{initial_2}
\widehat{\psi_N}(\xi)=N^{-s+2} \bigl( \chi_{B^{+}}(\xi)+ \chi_{B^{-}} (\xi) 
\bigr)
+N^{4a+2} \chi_{[\gamma/2,3\gamma/2]} (\xi).
\end{align}
Clearly $\|  \psi_N \|_{H^{s,a}} \sim 1$. Note $c_2 \neq c_3$ and $c_3 \neq 0$. Inserting (\ref{initial_2}) into (\ref{qua-Fo}), we have 
\begin{align*}
|\widehat{A}_{2,1}(\psi_N) (t)| \gtrsim 
N^{-2s+2} |\xi| \chi_{[0,\gamma]} (\xi) + \text{(remainder terms)},
\end{align*}
and 
\begin{align*}
|\widehat{A}_{2,2}(\psi_N ) (t) | \gtrsim N^{-s+4a+2} 
|\xi | \chi_{[N,N+\gamma]} (\xi) 
+ \text{(remainder terms)}.
\end{align*}
Therefore we obtain 
\begin{align} \label{qua-es_2}
\| A_2 (\psi_N) (t) \|_{H^{s,a}} \gtrsim N^{-2s+2} 
\Bigl( \int_{0}^{\gamma} |\xi|^{2a+2} d\xi \Bigr)^{1/2}
+ N^{-s+4a+2} 
\Bigl( \int_{N}^{N+\gamma} |\xi|^{2s+2} d\xi \Bigr)^{1/2}.
\end{align}
If $a \leq -3/2$, the first term of the right hand side of (\ref{qua-es_2}) diverges. When we assume $a>-3/2$, $\| A_2(\psi_N) (t)  \|_{H^{s,a}}$ is greater than $C (N^{-2(s+2a+2)}+ N^{4(a+1/4)})$. If $s< -2a-2$ or $a>-1/4$, 
$\| A_2 (\psi_N) (t)  \|_{H^{s,a}} \rightarrow \infty$ as $N \rightarrow \infty$, which implies the claim since $\| \psi_N \|_{H^{s,a}} \sim 1$.

\vspace{0.5em}

Finally, we prove Theorem~\ref{thm_ill} (iii). Similar to above, we seek for the initial data such that, for $|t|$ bounded,
\begin{align} \label{cub-es}
\| A_3(\phi_N) (t) \|_{H^{s,a}} \lesssim \| u_0 \|_{H^{s,a}}^3
\end{align}
fails. By using the similar argument to \cite{Bo97}, we prove that (\ref{cub-es}) fails for $s<-1/4$. $A_3(u_0)$ is the cubic term of the Taylor expansion of the flow map as follows: 
\begin{align*}
A_3(u_0)(t)= A_{3,1}(u_0)(t)+ A_{3,2}(u_0)(t)+ \text{(remainder terms)},
\end{align*}
where 
\begin{align*}
A_{3,1}(u_0)(t):= -c_1 \int_0^t U(t-s) \p_x (u_1(s))^3 ds,
\end{align*}
and 
\begin{align*}
A_{3,2}(u_0)(t):= -c_3 \int_0^t U(t-s) \p_x^3 (u_1(s) A_2(u_0)(s) ) ds.
\end{align*}
We put the initial data 
$\{ \phi_N \}_{N=1}^{\infty} \in H^{\infty}$ as follows:
\begin{align*}
\phi_N (x):= N^{-s+3/4} \cos(Nx) \int_{-N^{-3/2}}^{N^{-3/2}} e^{i \xi x} d\xi.
\end{align*} 
A simple calculation shows that
\begin{align} \label{initial_3}
\widehat{\phi_N}(\xi)=N^{-s+3/4} \bigl(\chi_{C^{+}}(\xi)
+\chi_{C^{-}} (\xi)  \bigr),
\end{align}
where $C^{\pm}:=[\pm N-N^{-3/2}, \pm N+N^{-3/2}]$. 
Clearly 
$\| \phi_N \|_{H^{s,a}} \sim 1$. 
A straightforward computation shows that
\begin{align} \label{cu_3-1}
A_{3,1}(u_0)(t) &=-c_1 \int \exp(i (\xi_1+ \xi_2+\xi_3)x+ i(\xi_1+\xi_2+\xi_3)^5 t) \nonumber \\
& \hspace{0.6cm} \times \frac{1-\exp(-i q_2 t)}{q_2}  (\xi_1+\xi_2+\xi_3) 
\widehat{u_0}(\xi_1) \widehat{u_0}(\xi_2) \widehat{u_0}(\xi_3) d\xi_1
d\xi_2 d\xi_3,
\end{align}
where 
\begin{align*}
q_2:= \frac{5}{2} (\xi_1+\xi_2) (\xi_2+\xi_3) (\xi_3+\xi_1) 
\bigl\{ (\xi_1+ \xi_2)^2 +(\xi_2+\xi_3)^2+ (\xi_3+\xi_1)^2 \bigr\}.
\end{align*}
Next we calculate $A_{3,2}(u_0)$. From the definition of the quadratic term $A_2$, 
\begin{align} \label{A_2-Fo}
&\widehat{A}_2(u_0)(t) =  \frac{2}{5}(c_3-c_2) \int \frac{\exp(i \xi_1^5 t+ i(\xi-\xi_1)^5 t )}{\xi^2 +\xi_1^2+(\xi-\xi_1)^2} \widehat{u}_0(\xi_1) \widehat{u}_0 (\xi-\xi_1) d\xi_1 \nonumber \\
& ~~- \frac{2}{5} (c_3-c_2)  \int \frac{\exp(i \xi^5 t)}{\xi^2+\xi_1^2 +(\xi-\xi_1)^2} \widehat{u}_0(\xi_1) \widehat{u}_0 (\xi-\xi_1) d\xi_1
+ \text{(remainder terms)}. 
\end{align}
Substituting (\ref{A_2-Fo}) into $A_{3,2}(u_0)$, we have
\begin{align} \label{cu_3-2}
& A_{3,2}(u_0)(t) =\frac{2}{5} c_3 (c_3-c_2) \int \exp(i (\xi_1+ \xi_2+\xi_3)x+ i(\xi_1+\xi_2+\xi_3)^5 t) 
 \frac{1-\exp(-i q_2 t)}{q_2} \nonumber \\ 
& \hspace{1.8cm} \times  
\frac{(\xi_1+\xi_2+\xi_3)^3}{\xi_2^2 +\xi_3^2+(\xi_2 +\xi_3)^2} 
\widehat{u_0}(\xi_1) \widehat{u_0}(\xi_2) \widehat{u_0}(\xi_3) d\xi_1
d\xi_2 d\xi_3 \nonumber \\
&  - \frac{2}{5}c_3(c_3-c_2) \int \exp(i (\xi_1+ \xi_2+\xi_3)x+ i(\xi_1+\xi_2+\xi_3)^5 t)  \frac{1-\exp(-i q_3 t)}{q_3} \nonumber \\
& \hspace{0.6cm} \times 
\frac{(\xi_1+\xi_2+\xi_3)^3}{\xi_2^2+\xi_3^2+(\xi_2+\xi_3)^2} 
\widehat{u_0}(\xi_1) \widehat{u_0}(\xi_2) \widehat{u_0}(\xi_3) d\xi_1
d\xi_2 d\xi_3 + \text{(remainder terms)}.
\end{align}
We assume that $\xi_1 \in C^{+}$, $\xi_2 \in C^{-}$ and $\xi_3 \in C^{+}$. 
Following 
(\ref{cu_3-1}) and (\ref{cu_3-2}), we have 
\begin{align} \label{cu_3-3}
&|\widehat{A}_{3}(\phi_N)(t)| \geq \Bigl| \exp (i \xi^5 t) \xi \nonumber \\
& \hspace{0.3cm} \times \int \Bigl\{ \Bigl(\frac{1}{5}c_{3}(c_3-c_2) -c_1\Bigr) \frac{1-\exp(-iq_2 t)}{q_2}-\frac{1}{5}(c_3-c_2) \frac{1-\exp(-i q_3 t)}{q_3} \Bigr\}  \nonumber \\
& \hspace{1.2cm} \times \widehat{\phi}_N(\xi_1) \widehat{\phi}_N(\xi_2) 
\widehat{\phi}_{N} (\xi-\xi_1-\xi_2) d\xi_1 d\xi_2 \Bigr| +\text{(remainder terms)}.
\end{align}
Here we used the change variables from $\xi_3$ to $\xi=\xi_1+\xi_2+\xi_3$. 
From $c_1 \neq \frac{1}{5} c_3(c_3-c_2)$ and 
(\ref{cu_3-3}), we obtain
\begin{align*}
\bigl| \widehat{A}_3(\phi_N) (t) \bigr|  \geq & \frac{|t|}{2}
\Bigl| c_1-\frac{1}{5} c_3 (c_3-c_2) \Bigr|~ N^{-3s-3/4} ~|\xi| \chi_{[N-N^{-3/2}, N+N^{-3/2}]} (\xi) \\
- & C N^{-3s-9/4} ~|\xi| 
\chi_{[N-N^{-3/2}, N+N^{-3/2}]} (\xi),
\end{align*}
where $C \geq 0$ is some constant. 
Thus there exists a constant $C'>0$ such that 
\begin{align*}
\| A_3(\phi_N) (t) \|_{H^{s,a}} \geq C' N^{-2s-1/2}
-2C N^{-2s-3}.
\end{align*}
Therefore, when $s<-1/4$ and $a \in \mathbb{R}$, there is no positive constant 
$C$ such that $
\| A_3 (\phi_N) (t) \|_{H^{s,a}} \leq C \| \phi_N \|_{H^{s,a}}^3$ for bounded $|t|$.

\section{Appendix}
We mention the typical counterexamples of (\ref{BE-3}) for (\ref{co_cr_1}).

\vspace{1em}

\noindent
\textbf{Example 1.} ($high \times high \rightarrow low$ interaction) \\
We define the rectangles $P_1, P_2$ as follows:
\begin{align*}
P_1:= & \bigl\{ (\tau,\xi) \in \mathbb{R}^2~ ;~ 
|\xi-N | \leq  N^{-3/2}, ~~ |\tau-( 5N^4 \xi -4N^5) |  \leq 1/2  \bigr\}, \\
P_2:= & \bigl\{ (\tau,\xi) \in \mathbb{R}^2~;~(-\tau,-\xi) \in A_1 \bigr\}. 
\end{align*}
Here we put
\begin{align} \label{rec-1}
f(\tau,\xi):= \chi_{P_1}(\tau,\xi), \hspace{0.3cm}
g(\tau,\xi):=\chi_{P_2}(\tau,\xi).
\end{align}
Then we have 
\begin{align} \label{int-1}
 f*g(\tau,\xi) \gtrsim  N^{-3/2}~\chi_{R_1} (\tau,\xi),
\end{align}
where
\begin{align*}
 R_1:=  \bigl\{ (\tau,\xi) \in \mathbb{R}^2~;~ 
\xi \in [ 1/2 N^{-3/2},3/4 N^{-3/2} ], ~~
 |\tau- 5N^4 \xi | \leq 1/2 \bigr\}.
\end{align*}
Inserting (\ref{rec-1}) and (\ref{int-1}) into (\ref{BE-3}), the necessary condition for (\ref{BE-3}) is $b \leq 3a/5+4s/5+11/10$. 
Thus $b\leq 3a/5+ 9/10$ if (\ref{BE-3}) for $s=-1/4$. \\
\newline
\textbf{Example 2.} ($high \times low \rightarrow high$ interaction) \\
We define the rectangle $Q$ as follows:
\begin{align*}
Q:=\bigl\{ (\tau,\xi) \in \mathbb{R}^2 ~;~|\xi-2N^{-3/2}| \leq N^{-3/2},~~
|\tau-(5N^4 \xi)| \leq 1/2 \bigr\}.
\end{align*}
Here we put 
\begin{align} \label{rec-2}
f(\tau,\xi) = \chi_{P_1}(\tau,\xi), \hspace{0.3cm} 
g(\tau,\xi)=\chi_{Q}(\tau,\xi).
\end{align}
Then we have 
\begin{align} \label{int-2}
f*g(\tau,\xi) \gtrsim N^{-3/2}~\chi_{R_2} (\tau,\xi),
\end{align}
where
\begin{align*}
R_2:=  \bigl\{ (\tau,\xi) \in \mathbb{R}^2~;~ 
|\xi-N| \leq N^{-3/2}/4 , ~~
 |\tau- (5N^4 \xi-4N^5 ) | \leq 1/2 \bigr\}.
\end{align*}
Substituting (\ref{rec-2}) and (\ref{int-2}) into (\ref{BE-3}), 
the necessary condition for (\ref{BE-3}) is $b \geq 3a/5+9/10$. \\ 
\newline
\textbf{Example 3.} ($high \times high \rightarrow high$ interaction) \\
We put  
\begin{align} \label{rec-3}
f(\tau,\xi)= \chi_{P_1}(\tau,\xi), \hspace{0.3cm} 
g(\tau,\xi)= \chi_{P_1}(\tau,\xi).
\end{align} 
Then we have
\begin{align} \label{int-3}
 f*g(\tau,\xi) \gtrsim  N^{-3/2}~\chi_{R_3} (\tau,\xi),
\end{align}
where
\begin{align*}
 R_3:=  \bigl\{ (\tau,\xi) \in \mathbb{R}^2~;~| \xi -2N | \leq N^{-3/2}/2 ,
~~ | \tau- (5N^4 \xi -8N^5) | \leq 1/2 \bigr\}.
\end{align*}
Inserting (\ref{rec-3}) and (\ref{int-3}) into (\ref{BE-3}), the necessary condition for (\ref{BE-3}) is 
$b \leq s/5+11/20$ for $s=-1/4$. 

On the other hand, we put 
\begin{align} \label{rec-4}
f(\tau,\xi)=\chi_{R_3}(\tau,\xi), \hspace{0.3cm}
g(\tau,\xi)=\chi_{P_2}(\tau,\xi). 
\end{align}
Then we have
\begin{align} \label{int-4}
f*g(\tau,\xi) \gtrsim N^{-3/2}~\chi_{R_2}(\tau,\xi).
\end{align}
Substituting (\ref{rec-4}) and (\ref{int-4}) into (\ref{BE-3}), 
the necessary condition for (\ref{BE-3}) is $b \geq 1/2$ for $s=-1/4$.

\end{document}